\documentclass{article}

\usepackage[nobottomtitles]{titlesec}

\usepackage{amsfonts,amsmath,amssymb,amsthm,array,booktabs,cite,color,graphbox,graphics,graphicx,hyperref, import,mathtools,pdfpages,stackengine,tikz-cd,url,varwidth,wrapfig}

\usepackage[margin=10pt,font=small,labelfont=bf,labelsep=period]{caption}

\usepackage[shortlabels ]{enumitem}
\usepackage[all]{xy}
\usepackage[nottoc]{tocbibind}

\newcommand{\e}{\varepsilon}

\newlength{\notewidth}
\newcommand{\note}[1]{\marginpar{\scriptsize\color{red}#1}}

\newcommand{\s}{\small}
\renewcommand{\ss}{\scriptsize}

\newcommand{\eps}{\varepsilon}

\newcommand{\sn}{\smallskip\noindent}
\newcommand{\mn}{\medskip\noindent}
\newcommand{\bn}{\bigskip\noindent}

\newcommand{\be}{\begin{equation}}
\newcommand{\ee}{\end{equation}}

\renewcommand{\iff}{if and only if\ }

\newcommand{\cL}{\mathcal L}

\newcommand{\R}{\mathbb{R}}

\newcommand{\RP}{\mathbb{RP}}

\newcommand{\g}{\gamma}

\newtheorem{theorem}{Theorem}
\newtheorem{prop}{Proposition}
\newtheorem{lemma}{Lemma}

\newtheorem{conjecture}{Conjecture}

\theoremstyle{definition}

\newtheorem{definition}{Definition}

\newtheorem{remark}{Remark}

\title{Cusps of caustics by reflection in ellipses}

\author{Gil Bor\footnote{
CIMAT, A.P. 402, Guanajuato, Gto. 36000, M\'exico; 
{\em gil@cimat.mx}
}
\and
Mark Spivakovsky\footnote{CNRS UMR 5219, Institut de Math\'matiques de Toulouse, 118, rte de Narbonne, 31062 Toulouse cedex 9, France and
Instituto de Matem\'aticas (Cuernavaca) LaSol, UMI CNRS 2001, UNAM, Av. Universidad s/n. Col. Lomas de Chamilpa, 62210, Cuernavaca, Morelos, M\'exico;
 {\em mark.spivakovsky@math.univ-toulouse.fr}
}
\and
Serge Tabachnikov\footnote{
Department of Mathematics,
Penn State University, 
University Park, PA 16802; USA
{\em tabachni@math.psu.edu}}
}

\date{\today}

\begin{document}
\maketitle
\tableofcontents

\newpage
\begin{abstract}
This paper is concerned with the billiard version of Jacobi's last geometric statement and its generalizations. Given a non-focal point $O$ inside an elliptic billiard table, one considers the family of rays emanating from $O$ and the caustic $ \Gamma_n$ of the reflected family after $n$ reflections off the ellipse, for each positive integer $n$.
It is known that $\Gamma_n$ has at least four cusps and it has been conjectured that it has exactly four (ordinary) cusps. The  present paper presents  a proof of this conjecture in the special case when the ellipse is a circle. In the case of an arbitrary ellipse, we give an explicit description of the location of four of the cusps of $\Gamma_n$, though we do not prove that these are the only cusps.

\end{abstract}

\section{Introduction and statement of results} \label{sec:intro}

The motivation for this work goes back to Jacobi's 1842-3  ``Lectures on Dynamics" \cite{Ja}. Recall that the conjugate locus of a point on a surface is the locus of the first conjugate points on the geodesics that start at this point. 
Jacobi considered the conjugate locus of a non-umbilic point on the surface of a triaxial  ellipsoid in 3-space. What is known as the {\it Last Geometric Statement of Jacobi} is the claim that this conjugate locus has exactly four cusps, see Figure \ref{surf}. We refer to \cite{Si} for a detailed historical discussion. 

\begin{figure}[ht]
\centering
\includegraphics[height=.18\textheight]{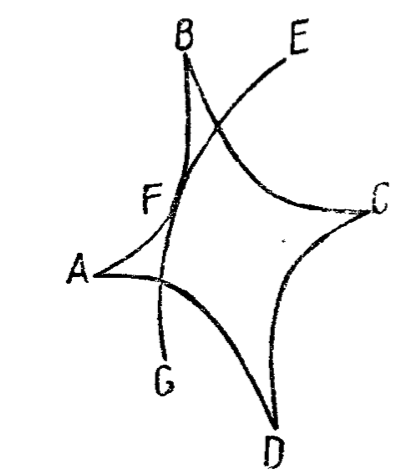}
\hspace{.2\textwidth}
\includegraphics[height=.18\textheight]{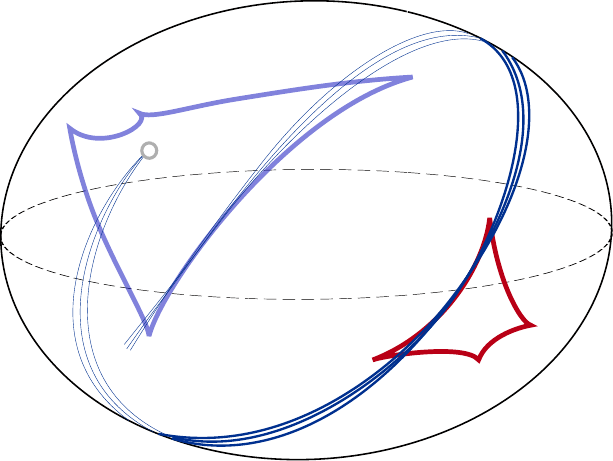}
\caption{Left: A sketch of the conjugate locus from \cite{Ja}. Right: The first (red) and second (blue) conjugate locus of a point on a triaxial ellipsoid.}
\label{surf}
\end{figure}

The Last Geometric Statement of Jacobi was proved only recently \cite{IK}. In contrast, it was known for a long time that the conjugate locus of a generic point on a convex surface has at least four cusps; see \cite{Bl} where this theorem is attributed to C. Carath\'eodory and \cite{Wa} for a recent proof.

The conjugate locus of a point is also called the first caustic. One considers the loci of the second, third, etc., conjugate points on the geodesics emanating from a point; these are the second, third, etc., caustics. These curves are also the components of the envelope of the 1-parameter family of geodesics that start at this point. Figure \ref{surf} (right) depicts the first and second such caustics. 

This article concerns the billiard versions of these problems. G. Birkhoff \cite{Bi} suggested to consider billiard trajectories in a convex plane domain as the geodesics on a ``pancake", the surface  obtained from the domain by infinitesimally ``thickening" it. This leads to the following set-up.

Consider an oval $C$, a smooth strictly convex closed curve in the plane, the boundary of a billiard table. Let $O$ be a point inside $C$ and consider the billiard trajectories that start at $O$. After  $n$ reflections off $C$, we obtain a 1-parameter family of lines whose envelope is  a closed connected curve in the real projective plane $\RP^2$, possibly with some cusps and self intersections,  called the $n$-th {\it caustic by reflection} from $O$. The term caustic, meaning ``capable of burning,"  comes from optics, where  $C$ is an ideal mirror and  $O$ is a  light source. 

We refer to \cite{Bo,BGG} and the literature cited therein for the study of the first caustics by reflection, also known as {\em catacaustics}. In particular, A. Cayley studied the first caustics by reflection and refraction in a circle in his memoire \cite{Ca} where he considered the cases when the source of light was inside the circle, on the circle, and outside the circle, including at infinity. 

We proved in \cite{BT} that, for every $n\geq 1$, if $O$ is a generic point inside an oval, then the $n$-th caustic by reflection from $O$ has at least 4 cusps. This is one of many variations on the classic 4-vertex theorem.  Here are refined versions of two conjectures made in \cite{BT}.

\begin{conjecture} \label{conj} If $C$ is an ellipse and $O$ is an interior point which is not a focus of $C$  then, for all $n\geq 1$, the $n$-th caustic by reflection from $O$ has exactly four cusps, and all four are ordinary ones. \end{conjecture}

See the Section \ref{sec:dual} for a precise definition of 
``ordinary cusp".

\begin{remark}\label{remark:limit}  The $n=1$ case of Conjecture \ref{conj} (without the ``ordinary" part)  can be thought of  as a  ``limiting case" of the Jacobi's Last Geometric Statement, as one of the axes of the ellipsoid tends to 0. 
\end{remark}

\begin{conjecture} If an oval $C$ is not an ellipse then there exists an $n\geq 1$ and an open set $U$ inside $C$ such that for every $O\in U$ the number of cusps of the 
$n$-th caustic by reflection from $O$  is greater than four. 
\end{conjecture}

An analogue of Conjecture \ref{conj} for the caustics of geodesics emanating from a point on a tri-axial  ellipsoid was experimentally studied in \cite{Si}. That paper contains numerous computer generated images of first, second, third, and fourth caustics, each having exactly four cusps.

\sn

This article is a step toward proving Conjecture \ref{conj}. To state our  first result, we recall a well known  property of billiards in an ellipse. 

\sn 

An ellipse $C$ defines two 1-parameter families of {\em confocal conics}, those conics which share their foci with $C$. One family consists of ellipses, the other of hyperbolas (including  the major and minor axes of $C$). They  form, in the complement of the foci of $C$, a double foliation so that through each point  pass one confocal ellipse and one confocal hyperbola, intersecting orthogonally at the point.   A ray (directed line), incident to the interior  of $C$, is tangent to exactly one of these confocal conics (or incident to one of  the foci), and after reflection off $C$ it is tangent to the same conic. See Figure \ref{thm1} (left).

\begin{theorem} \label{thm:one}

Let $O$ be a non-focal point inside an ellipse $C$, and let $E$ and $H$ be the  ellipse and hyperbola (respectively), passing through $O$ and confocal to $C$. Consider the four rays emanating from $O$ and tangent to  $E$ and $H$ (two each). Then after $n$ reflections, the 4  rays are   tangent to $E$ and $H$ at 4 points which are cusps of the $n$-th caustic by reflection from $O$. 
See Figure \ref{thm1} (right).
\end{theorem}

\begin{figure}[ht]
\centering
\def\svgwidth{.32\textwidth}
\begingroup%
  \makeatletter%
  \providecommand\color[2][]{%
    \errmessage{(Inkscape) Color is used for the text in Inkscape, but the package 'color.sty' is not loaded}%
    \renewcommand\color[2][]{}%
  }%
  \providecommand\transparent[1]{%
    \errmessage{(Inkscape) Transparency is used (non-zero) for the text in Inkscape, but the package 'transparent.sty' is not loaded}%
    \renewcommand\transparent[1]{}%
  }%
  \providecommand\rotatebox[2]{#2}%
  \newcommand*\fsize{\dimexpr\f@size pt\relax}%
  \newcommand*\lineheight[1]{\fontsize{\fsize}{#1\fsize}\selectfont}%
  \ifx\svgwidth\undefined%
    \setlength{\unitlength}{314.51623535bp}%
    \ifx\svgscale\undefined%
      \relax%
    \else%
      \setlength{\unitlength}{\unitlength * \real{\svgscale}}%
    \fi%
  \else%
    \setlength{\unitlength}{\svgwidth}%
  \fi%
  \global\let\svgwidth\undefined%
  \global\let\svgscale\undefined%
  \makeatother%
  \begin{picture}(1,0.95931848)%
    \lineheight{1}%
    \setlength\tabcolsep{0pt}%
    \put(0,0){\includegraphics[width=\unitlength,page=1]{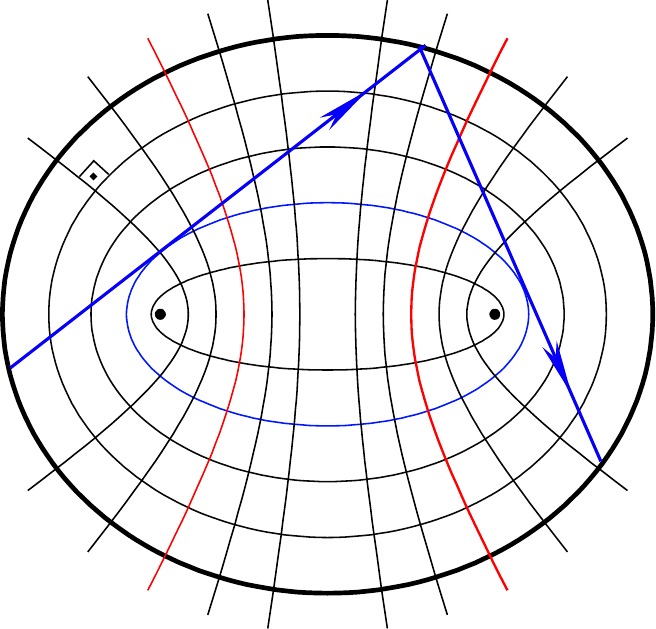}}%
    \put(0.90308033,0.12471819){\color[rgb]{0,0,0}\makebox(0,0)[lt]{\lineheight{1.25}\smash{\begin{tabular}[t]{l}\textit{$C$}\end{tabular}}}}%
    \put(0,0){\includegraphics[width=\unitlength,page=2]{confocal.pdf}}%
  \end{picture}%
\endgroup%

\hspace{.15\textwidth}
\def\svgwidth{.35\textwidth}
\begingroup%
  \makeatletter%
  \providecommand\color[2][]{%
    \errmessage{(Inkscape) Color is used for the text in Inkscape, but the package 'color.sty' is not loaded}%
    \renewcommand\color[2][]{}%
  }%
  \providecommand\transparent[1]{%
    \errmessage{(Inkscape) Transparency is used (non-zero) for the text in Inkscape, but the package 'transparent.sty' is not loaded}%
    \renewcommand\transparent[1]{}%
  }%
  \providecommand\rotatebox[2]{#2}%
  \newcommand*\fsize{\dimexpr\f@size pt\relax}%
  \newcommand*\lineheight[1]{\fontsize{\fsize}{#1\fsize}\selectfont}%
  \ifx\svgwidth\undefined%
    \setlength{\unitlength}{288bp}%
    \ifx\svgscale\undefined%
      \relax%
    \else%
      \setlength{\unitlength}{\unitlength * \real{\svgscale}}%
    \fi%
  \else%
    \setlength{\unitlength}{\svgwidth}%
  \fi%
  \global\let\svgwidth\undefined%
  \global\let\svgscale\undefined%
  \makeatother%
  \begin{picture}(1,0.8984375)%
    \lineheight{1}%
    \setlength\tabcolsep{0pt}%
    \put(0,0){\includegraphics[width=\unitlength,page=1]{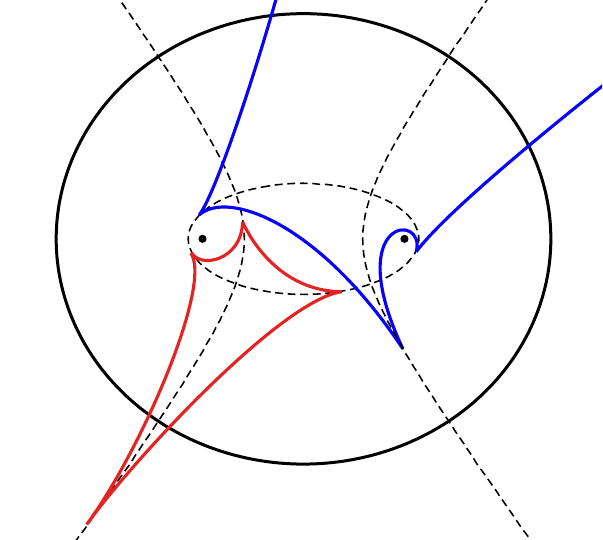}}%
    \put(0.64185862,0.58159802){\makebox(0,0)[lt]{\lineheight{1.25}\smash{\begin{tabular}[t]{l}$O$\end{tabular}}}}%
    \put(0,0){\includegraphics[width=\unitlength,page=2]{cusps3e.pdf}}%
    \put(0.10984438,0.81327849){\makebox(0,0)[lt]{\lineheight{1.25}\smash{\begin{tabular}[t]{l}$H$\end{tabular}}}}%
    \put(0,0){\includegraphics[width=\unitlength,page=3]{cusps3e.pdf}}%
    \put(0.7988588,0.81757141){\makebox(0,0)[lt]{\lineheight{1.25}\smash{\begin{tabular}[t]{l}$H$\end{tabular}}}}%
    \put(0.40605621,0.60936455){\makebox(0,0)[lt]{\lineheight{1.25}\smash{\begin{tabular}[t]{l}$E$\end{tabular}}}}%
    \put(0.82032339,0.18436503){\makebox(0,0)[lt]{\lineheight{1.25}\smash{\begin{tabular}[t]{l}$C$\end{tabular}}}}%
    \put(0,0){\includegraphics[width=\unitlength,page=4]{cusps3e.pdf}}%
  \end{picture}%
\endgroup%

\caption{Left: a light ray reflected off an elliptical table $C$ stays  tangent to the same confocal conic, either an ellipse (blue) or a hyperbola (red). Right: the first (red) and the second (blue) caustics by reflection from $O$ each have cusps at the four tangency points with the confocal conics through $O$ of the four reflected rays emanating from $O$ tangent to these conics.}
\label{thm1}
\end{figure}

\begin{remark} \label{remarks:thm1}
\begin{enumerate}[(a)]
 \item\label{it:axes} If $O$ lies on one of the axes of $C$, 
 then the role of $H$ in the above theorem is played  by this axis. The location of the corresponding cusps along this axis is then determined by the  ``mirror equation" of geometric optics. See Section \ref{sec:axes} below. 

\item The limiting case when $C$ is a circle is not excluded: in this case, the role of the two confocal conics through $O$ are played by the concentric circle through $O$ and the line  through $O$ and the center. See Figure \ref{fig:circle} (left).

\begin{figure}[ht]
\centering
\def\svgwidth{.5\textwidth}
\begingroup%
  \makeatletter%
  \providecommand\color[2][]{%
    \errmessage{(Inkscape) Color is used for the text in Inkscape, but the package 'color.sty' is not loaded}%
    \renewcommand\color[2][]{}%
  }%
  \providecommand\transparent[1]{%
    \errmessage{(Inkscape) Transparency is used (non-zero) for the text in Inkscape, but the package 'transparent.sty' is not loaded}%
    \renewcommand\transparent[1]{}%
  }%
  \providecommand\rotatebox[2]{#2}%
  \newcommand*\fsize{\dimexpr\f@size pt\relax}%
  \newcommand*\lineheight[1]{\fontsize{\fsize}{#1\fsize}\selectfont}%
  \ifx\svgwidth\undefined%
    \setlength{\unitlength}{337.99999237bp}%
    \ifx\svgscale\undefined%
      \relax%
    \else%
      \setlength{\unitlength}{\unitlength * \real{\svgscale}}%
    \fi%
  \else%
    \setlength{\unitlength}{\svgwidth}%
  \fi%
  \global\let\svgwidth\undefined%
  \global\let\svgscale\undefined%
  \makeatother%
  \begin{picture}(1,0.57988169)%
    \lineheight{1}%
    \setlength\tabcolsep{0pt}%
    \put(0,0){\includegraphics[width=\unitlength,page=1]{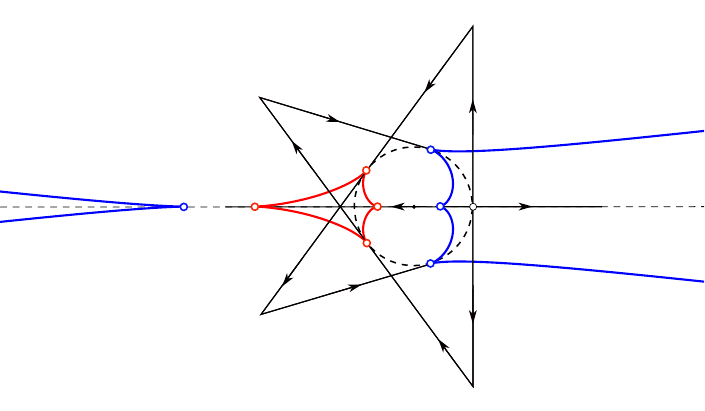}}%
    \put(0.68141014,0.29687498){\makebox(0,0)[lt]{\lineheight{1.25}\smash{\begin{tabular}[t]{l}$O$\end{tabular}}}}%
    \put(0,0){\includegraphics[width=\unitlength,page=2]{caustic_circle2.pdf}}%
    \put(0.81373882,0.07445027){\makebox(0,0)[lt]{\lineheight{1.25}\smash{\begin{tabular}[t]{l}$C$\end{tabular}}}}%
  \end{picture}%
\endgroup%

\hspace{.1\textwidth}
\includegraphics[width=.37\textwidth]{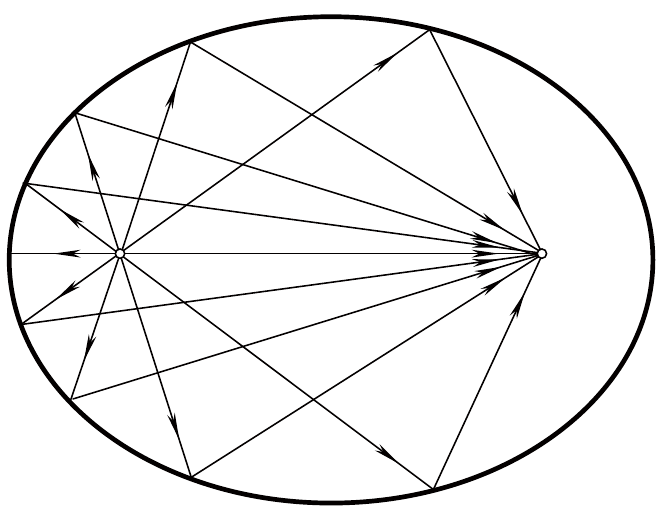}
\caption{Left: Theorem \ref{thm:one} for a circle. Right: The $n$-th caustic from a focus of an ellipse is the other focus for odd $n$, the same focus for even $n$.  }
\label{fig:circle}
\end{figure}

\item
 If the point $O$ is a focus of the ellipse, then the $n$-th caustic by reflection degenerates to one of the two foci, depending on the parity of $n$. See Figure \ref{fig:circle} (right).

\item 
The stated location of the 4 cusps in Theorem \ref{thm:one} 
can be deduced from the conjectures made
 in \cite{Si} about the location of the cusps of  caustics of envelopes of geodesics from a point on an ellipsoid.

\item  It is straightforward to extend Theorem \ref{thm:one} to an arbitrary non-degenerate conic section $C$  (parabola and hyperbola). 
The complement of the closure of $C$ in $\RP^2$ consists of two components, diffeomorphic to a disc and to a M\"obius band, respectively. The former can serve as a billiard table, and our proof of Theorem \ref{thm:one} applies, mutatis mutandis, to it as well. See Figure \ref{hyperbola}.  

\begin{figure}[ht]
\centering
\def\svgwidth{.9\textwidth}
\begingroup%
  \makeatletter%
  \providecommand\color[2][]{%
    \errmessage{(Inkscape) Color is used for the text in Inkscape, but the package 'color.sty' is not loaded}%
    \renewcommand\color[2][]{}%
  }%
  \providecommand\transparent[1]{%
    \errmessage{(Inkscape) Transparency is used (non-zero) for the text in Inkscape, but the package 'transparent.sty' is not loaded}%
    \renewcommand\transparent[1]{}%
  }%
  \providecommand\rotatebox[2]{#2}%
  \newcommand*\fsize{\dimexpr\f@size pt\relax}%
  \newcommand*\lineheight[1]{\fontsize{\fsize}{#1\fsize}\selectfont}%
  \ifx\svgwidth\undefined%
    \setlength{\unitlength}{851.96200562bp}%
    \ifx\svgscale\undefined%
      \relax%
    \else%
      \setlength{\unitlength}{\unitlength * \real{\svgscale}}%
    \fi%
  \else%
    \setlength{\unitlength}{\svgwidth}%
  \fi%
  \global\let\svgwidth\undefined%
  \global\let\svgscale\undefined%
  \makeatother%
  \begin{picture}(1,0.35938721)%
    \lineheight{1}%
    \setlength\tabcolsep{0pt}%
    \put(0,0){\includegraphics[width=\unitlength,page=1]{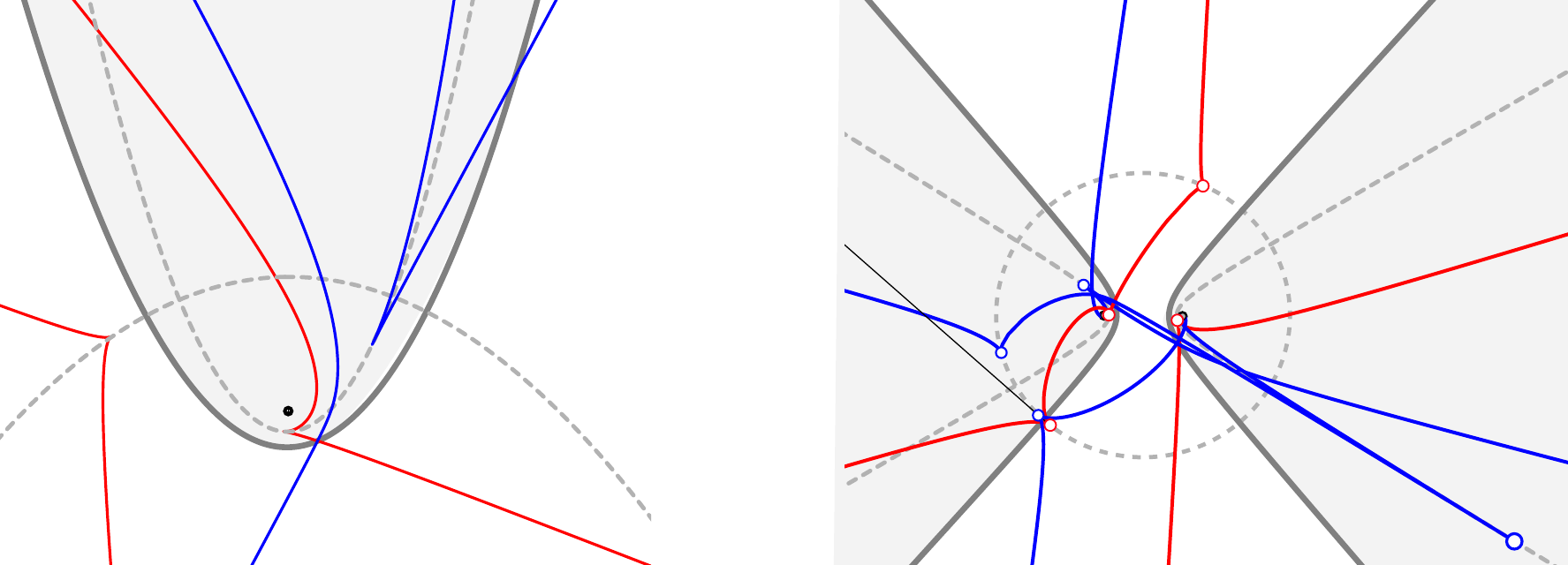}}%
    \put(0.13132685,0.14841566){\makebox(0,0)[lt]{\lineheight{1.25}\smash{\begin{tabular}[t]{l}$O$\end{tabular}}}}%
    \put(0,0){\includegraphics[width=\unitlength,page=2]{caustics_parabola_hyperbola.pdf}}%
    \put(0.82801369,0.18939725){\makebox(0,0)[lt]{\lineheight{1.25}\smash{\begin{tabular}[t]{l}$O$\end{tabular}}}}%
    \put(0,0){\includegraphics[width=\unitlength,page=3]{caustics_parabola_hyperbola.pdf}}%
  \end{picture}%
\endgroup%

\caption{Theorem \ref{thm:one} holds for  any convex billiard table in the projective plane, bounded by a conic; shown are a parabola (left) and a hyperbola (right). Some cusps are out of sight. }
\label{hyperbola}
\end{figure}

\item 
In Section \ref{sec:Lio} below Theorem \ref{thm:one} is further extended to ``Liouville billiards", where the billiard table is formed by a coordinate line on a Liouville surface.\\

\end{enumerate}\end{remark}

Thus, after Theorem \ref{thm:one}, proving  Conjecture \ref{conj} amounts to showing that the 4 cusps described  by  Theorem \ref{thm:one} are the {\em only} cusps of the $n$-th caustic by reflection from $O$ and that all 4 cusps are ordinary. We were able to show this only in the case when $C$ is a circle, which is our next result.

\begin{theorem} \label{thm:two}
Conjecture \ref{conj} holds if  $C$ is a circle. Namely,  if  $O$ is an interior point of a circle $C$, different from its center, then for every $n\geq 1$  there are exactly  4 cusps on the $n$-th caustic by reflection from $O$; two of these cusps lie on the line passing through $O$ and  the center of the circle, the other two on the  circle  through $O$ concentric with $C$.  Furthermore, these 4 cusps are ordinary. 
\end{theorem}

The content of this article is as follows. In Section \ref{sec:prelim} we recall relevant facts about billiards in ellipses, envelopes of families of lines and their cusps.
In Section \ref{sec:pfone} we prove Theorem \ref{thm:one} and in Section \ref{sec:pftwo} we prove Theorem \ref{thm:two}.  Section \ref{sec:misc} contains various additional results and suggested problems. 

\bn{\bf Acknowledgements}. GB acknowledges hospitality of the Toulouse Mathematics Institute during visits in 2023-4 and a CONAHCYT Grant A1-S-45886.
ST participated in the special program ``Mathematical Billiards: at the Crossroads of Dynamics, Geometry, Analysis, and Mathematical Physics" at Simons Center for Geometry and Physics, where this work started; he is grateful to the Center for its hospitality and the inspiring atmosphere. 
ST was supported by NSF grants DMS-2005444 and DMS-2404535.

\section{Preliminaries} \label{sec:prelim}
\subsection{Billiards in ellipses} \label{sec:bil}

Let us recall relevant  facts concerning  billiards in ellipses, in particular, their complete integrability, see, e.g., \cite{GIT,IT}, or \cite{Ta}.

Consider a billiard table $C$ bounded by an ellipse
$$
\frac{x^2}{a^2} + \frac{y^2}{b^2}=1, \quad \mbox{ where }0<b\leq a.
$$

Associated with the billiard table $C$ is a dynamical system whose phase space $\cL$ (topologically a cylinder) is the space of rays (oriented lines) that intersect the interior of $C$. The billiard transformation $T$ is the transformation of the phase space that  sends an incoming ray to the outgoing one upon reflection  off $C$. See Figure \ref{space} (left). 

The phase cylinder $\cL$  admits a $T$-invariant area form. If a ray is characterized by its direction $\alpha$ and the signed distance from the origin  $p$ (see Figure \ref{fig:ap}), then the area form is $dp \wedge d\alpha$. This fact is not specific to ellipses: this area form is invariant under the billiard transformation in a billiard table of any  shape.

\begin{figure}[ht]
\centering
\def\svgwidth{.3\textwidth}
\begingroup%
  \makeatletter%
  \providecommand\color[2][]{%
    \errmessage{(Inkscape) Color is used for the text in Inkscape, but the package 'color.sty' is not loaded}%
    \renewcommand\color[2][]{}%
  }%
  \providecommand\transparent[1]{%
    \errmessage{(Inkscape) Transparency is used (non-zero) for the text in Inkscape, but the package 'transparent.sty' is not loaded}%
    \renewcommand\transparent[1]{}%
  }%
  \providecommand\rotatebox[2]{#2}%
  \newcommand*\fsize{\dimexpr\f@size pt\relax}%
  \newcommand*\lineheight[1]{\fontsize{\fsize}{#1\fsize}\selectfont}%
  \ifx\svgwidth\undefined%
    \setlength{\unitlength}{458.75024414bp}%
    \ifx\svgscale\undefined%
      \relax%
    \else%
      \setlength{\unitlength}{\unitlength * \real{\svgscale}}%
    \fi%
  \else%
    \setlength{\unitlength}{\svgwidth}%
  \fi%
  \global\let\svgwidth\undefined%
  \global\let\svgscale\undefined%
  \makeatother%
  \begin{picture}(1,0.74851136)%
    \lineheight{1}%
    \setlength\tabcolsep{0pt}%
    \put(0.6659956,0.33437928){\color[rgb]{0,0,0}\makebox(0,0)[lt]{\lineheight{1.25}\smash{\begin{tabular}[t]{l}$p>0$\end{tabular}}}}%
    \put(0.20774027,0.55463674){\color[rgb]{0.03921569,0,0}\makebox(0,0)[lt]{\lineheight{1.25}\smash{\begin{tabular}[t]{l}$p<0$\end{tabular}}}}%
    \put(0.38394375,0.27868975){\color[rgb]{0,0,0}\makebox(0,0)[lt]{\lineheight{1.25}\smash{\begin{tabular}[t]{l}$O$\end{tabular}}}}%
    \put(0.22570032,0.72838194){\color[rgb]{0,0,0}\makebox(0,0)[lt]{\lineheight{1.25}\smash{\begin{tabular}[t]{l}$\alpha$\end{tabular}}}}%
    \put(0,0){\includegraphics[width=\unitlength,page=1]{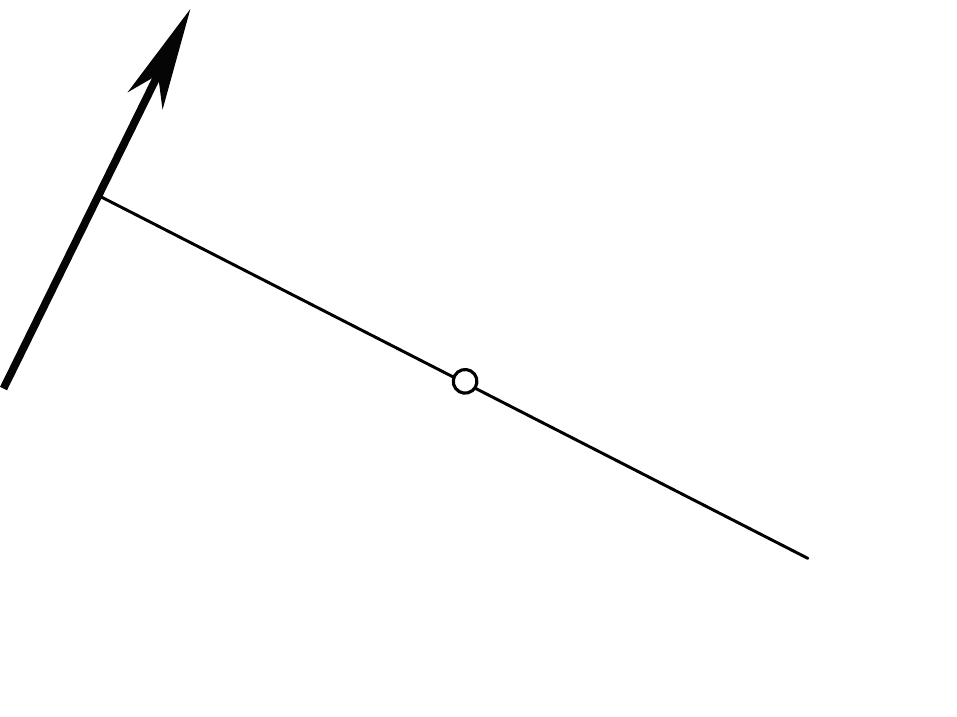}}%
    \put(0.51806246,0.66913338){\color[rgb]{1,1,1}\rotatebox{37.905915}{\makebox(0,0)[lt]{\lineheight{1.25}\smash{\begin{tabular}[t]{l}$alpha$\end{tabular}}}}}%
    \put(0.94166961,0.42205347){\color[rgb]{0,0,0}\makebox(0,0)[lt]{\lineheight{1.25}\smash{\begin{tabular}[t]{l}$\alpha$\end{tabular}}}}%
    \put(0,0){\includegraphics[width=\unitlength,page=2]{rays2.pdf}}%
  \end{picture}%
\endgroup%
\hspace{.07\textwidth}
\caption{The coordinates $(\alpha, p)$ on the space of oriented lines. 
}
\label{fig:ap}
\end{figure}

The ellipse $C$ is included in a confocal family of conics
$$
C_\lambda: \frac{x^2}{a^2-\lambda} + \frac{y^2}{b^2-\lambda}=1,\quad \lambda \in (-\infty, b^2)\cup(b^2, a^2).
$$
This  is an ellipse for $\lambda<b^2$ and a hyperbola for $b^2<\lambda<a^2$.
For $0 <\lambda<b^2$ the confocal ellipse $C_\lambda$ is contained in the interior of $C$,  for $\lambda<0$ it  is contained in the exterior of $C$. For $\lambda=0$ one has  $C_0=C$. 

As $\lambda$ tends to $b^2$ on  the left, the confocal ellipse $C_\lambda$ 
 tends to the line segment on the $x$-axis connecting the two foci of $C$; 
  the right limit is the closure of the complement of this segment in the $x$-axis. 
 As $\lambda$ tends to $a^2$ on the left, $C_\lambda$ tends to the $y$-axis.

 A  ray $r\in\cL$, not incident to one of the foci of $C$,   is tangent to a unique conic 
 $C_\lambda$ from this confocal family, so $\lambda$ can be considered as a function on $\cL$. As is easy to show, it is  given by  
$$\lambda=(a\sin\alpha)^2+(b\cos\alpha)^2-p^2. 
$$
(This formula shows that $\lambda$  extends smoothly to all of $\cL$, including rays  incident to the foci.) 

After reflection, the  ray $T(r)$ 
is tangent to the same conic \cite[Theorem 4.4]{Ta}. 
Thus the level curves of $\lambda$  define a (singular) $T$-invariant foliation of the phase space ${\mathcal L}$, whose leaves  consist of the rays tangent to a fixed conic, see Figure \ref{space} (right). 

Note that the resulting foliation is non-singular away from the 4 marked points on the $\alpha$-axis (the critical points of $\lambda$), 
corresponding to the rays aligned with the major and minor axes of $C$. 
Note also that each   level curve of a regular value $\lambda\in(0,b^2)\cup(b^2,a^2)$ 
has {\em two}  connected components. For $\lambda\in (0,b^2)$ (rays tangent to a fixed confocal ellipse), 
each of the two components is $T$-invariant. For $\lambda\in  (b^2, a^2)$ (rays tangent to a fixed confocal hyperbola,
 including its asymptots), the two components are interchanged by $T$. 
 The figure $\infty$  (the level curve $\lambda=b^2$) corresponds to rays passing through the foci.  Orientation reversing acts on $\cL$  
 by $R:(\alpha,p)\mapsto(\alpha+\pi, -p)$, satisfying $R^2=(RT)^2=id$. The two reflections about the major and minor  axes of $C$ induce maps of $\cL$ commuting with $T$. 

\begin{figure}[ht]
\centering
\def\svgwidth{.4\textwidth}
\begingroup%
  \makeatletter%
  \providecommand\color[2][]{%
    \errmessage{(Inkscape) Color is used for the text in Inkscape, but the package 'color.sty' is not loaded}%
    \renewcommand\color[2][]{}%
  }%
  \providecommand\transparent[1]{%
    \errmessage{(Inkscape) Transparency is used (non-zero) for the text in Inkscape, but the package 'transparent.sty' is not loaded}%
    \renewcommand\transparent[1]{}%
  }%
  \providecommand\rotatebox[2]{#2}%
  \newcommand*\fsize{\dimexpr\f@size pt\relax}%
  \newcommand*\lineheight[1]{\fontsize{\fsize}{#1\fsize}\selectfont}%
  \ifx\svgwidth\undefined%
    \setlength{\unitlength}{340.99999237bp}%
    \ifx\svgscale\undefined%
      \relax%
    \else%
      \setlength{\unitlength}{\unitlength * \real{\svgscale}}%
    \fi%
  \else%
    \setlength{\unitlength}{\svgwidth}%
  \fi%
  \global\let\svgwidth\undefined%
  \global\let\svgscale\undefined%
  \makeatother%
  \begin{picture}(1,0.82991206)%
    \lineheight{1}%
    \setlength\tabcolsep{0pt}%
    \put(0,0){\includegraphics[width=\unitlength,page=1]{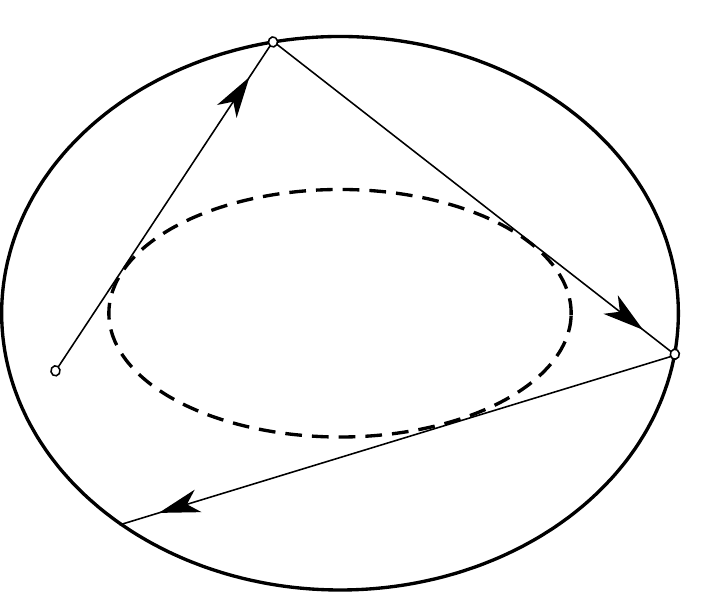}}%
    \put(0.07723372,0.25906167){\makebox(0,0)[lt]{\lineheight{1.25}\smash{\begin{tabular}[t]{l}\s$O$\\\end{tabular}}}}%
    \put(0.84860065,0.07477261){\makebox(0,0)[lt]{\lineheight{1.25}\smash{\begin{tabular}[t]{l}\s$C$\end{tabular}}}}%
    \put(0,0){\includegraphics[width=\unitlength,page=2]{conf.pdf}}%
    \put(0.22215421,0.60238247){\makebox(0,0)[lt]{\lineheight{1.25}\smash{\begin{tabular}[t]{l}\s$r$\end{tabular}}}}%
    \put(0.62447082,0.61749722){\makebox(0,0)[lt]{\lineheight{1.25}\smash{\begin{tabular}[t]{l}\s$T(r)$\end{tabular}}}}%
    \put(0,0){\includegraphics[width=\unitlength,page=3]{conf.pdf}}%
    \put(0.31618921,0.26544175){\makebox(0,0)[lt]{\lineheight{1.25}\smash{\begin{tabular}[t]{l}\s$C_\lambda$\end{tabular}}}}%
    \put(0.39757754,0.10991126){\makebox(0,0)[lt]{\lineheight{1.25}\smash{\begin{tabular}[t]{l}\s$T^2(r)$\end{tabular}}}}%
  \end{picture}%
\endgroup%
\hspace{.07\textwidth}
\hspace{.05\textwidth}
\def\svgwidth{.4\textwidth}
\begingroup%
  \makeatletter%
  \providecommand\color[2][]{%
    \errmessage{(Inkscape) Color is used for the text in Inkscape, but the package 'color.sty' is not loaded}%
    \renewcommand\color[2][]{}%
  }%
  \providecommand\transparent[1]{%
    \errmessage{(Inkscape) Transparency is used (non-zero) for the text in Inkscape, but the package 'transparent.sty' is not loaded}%
    \renewcommand\transparent[1]{}%
  }%
  \providecommand\rotatebox[2]{#2}%
  \newcommand*\fsize{\dimexpr\f@size pt\relax}%
  \newcommand*\lineheight[1]{\fontsize{\fsize}{#1\fsize}\selectfont}%
  \ifx\svgwidth\undefined%
    \setlength{\unitlength}{187.51748153bp}%
    \ifx\svgscale\undefined%
      \relax%
    \else%
      \setlength{\unitlength}{\unitlength * \real{\svgscale}}%
    \fi%
  \else%
    \setlength{\unitlength}{\svgwidth}%
  \fi%
  \global\let\svgwidth\undefined%
  \global\let\svgscale\undefined%
  \makeatother%
  \begin{picture}(1,0.98173729)%
    \lineheight{1}%
    \setlength\tabcolsep{0pt}%
    \put(0,0){\includegraphics[width=\unitlength,page=1]{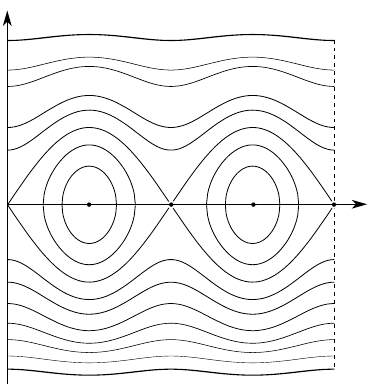}}%
    \put(0.081565,0.9400485){\color[rgb]{0,0,0}\makebox(0,0)[lt]{\lineheight{1.25}\smash{\begin{tabular}[t]{l}\s$p$\end{tabular}}}}%
    \put(0.87733188,0.51173214){\color[rgb]{0,0,0}\makebox(0,0)[lt]{\lineheight{1.25}\smash{\begin{tabular}[t]{l}\s$\alpha$\end{tabular}}}}%
    \put(0,0){\includegraphics[width=\unitlength,page=2]{space3.pdf}}%
  \end{picture}%
\endgroup%

\caption{Left: a billiards trajectory  in an ellipse and the associated confocal conic. Right: the phase space $\cL$ of the billiard transformation $T$  in an elliptical table (topologically a cylinder), and  its $T$-invariant foliation, which is regular away from the  4 marked points on the $\alpha$ axis, corresponding to rays aligned with the major and minor  axis of the table. 
Reversing the orientation of a ray corresponds to  the `glide-reflection' $(\alpha,p)\mapsto(\alpha+\pi, -p)$. The $\infty$-shaped curve corresponds to rays incident to the foci of the ellipse, phase curves inside it correspond to rays tangent to confocal hyperbolas (including their asymptotes), phase curves outside it to rays tangent to confocal ellipses.  }
\label{space}
\end{figure}

The following proposition is a special case of the Arnold-Liouville theorem on completely integrable Hamiltonian systems \cite{Ve}.
We shall give  a self-contained proof in our case, following  Chapter 4 of \cite{Ta}.

\begin{prop}
\label{prop:int}
On each leaf $\g$ of the $T$-invariant  foliation of $\cL$ there is a $T$-invariant non-vanishing 1-form, well defined up to multiplicative constant. Consequently, there is a local coordinate $t$ on $\g$ in which $T$ is given by $T(t)=t+c$ for some constant $c$. 
\end{prop}
\begin{proof} 

Choose a  smooth function $f$ without critical points in a neighborhood of $\g$, which is constant on each leaf of the $T$-invariant foliation (for example, $f=\lambda$). Then $f \circ T=f$ implies  $T^*df=df.$ Let $X_f$ be the Hamiltonian vector field associated to $f,$ that is,  $\omega(X_f, \,\cdot\,)=df,$ where $\omega=dp\wedge d\alpha $ is the $T$-invariant area form on $\cL$. Since both $df$ and $\omega$ are $T$-invariant, the same holds for  $X_f$. Since  $X_f$ is non-vanishing and tangent to $\gamma$, there is  a unique 1-form $\alpha$ on $\gamma$ such that $\alpha(X_f)=1.$ Since $X_f$ is $T$-invariant, so is $\alpha$. In neighborhoods of a point $r\in \gamma$ and its image $T(r)$, one can find coordinates $t$ and $t_1$, respectively,  such  that $\alpha=dt$ near $r$ and $\alpha=dt_1$ near $T(r)$. It follows that  $0=T^*\alpha-\alpha=d(t_1\circ T -t)$, thus $t_1\circ T -t=c$ for some constant $c$;  that is, $T$ is given near $r$   by $t_1=t+c.$ 

If one replaces $f$ by another function, say  $g=\phi(f),$ then the corresponding vector field  changes to 
$X_g = (\phi'\circ f)X_f$, i.e., a non-zero constant multiple of $X_f$ along $\g$, so $\alpha$ is also changed by a constant multiple.  
 \end{proof}
 
 Once a choice of coordinate $t$ is made on each level curve of $\lambda$,  one can use $(t,\lambda)$ as coordinates on ${\mathcal L}$ (away from singular leaves);  the ray $r(t,\lambda)$ is tangent to the confocal conic corresponding to the parameter $\lambda$, such that  $T(r(t,\lambda))=r(t+c(\lambda),\lambda).$

\begin{remark}\label{remark:ind}
It is important to note that the choice of the $t$ coordinate in the last proposition depends only on the $T$-invariant foliation of  $\cL$, which in turn depends on the  family of conics confocal to the billiard table $C$,  and not on a particular choice of conic within this family as a billiard table. That is,  if one chooses, as a billiard table, any conic confocal to $C$, say $C_1$, then the associated  billiard map $T_1$  with respect to $C_1$ admits the same invariant foliation of $\cL$ as $T$, and  is thus given in each invariant leaf by the same formula, $T_1(r(t,\lambda))=r(t+c_1(\lambda), \lambda).$ 

For example, consider an ellipse $E$ from a confocal family. One may think of  $t$ as a coordinate on $E$. Then the locus of the intersection points of the tangent to $E$, whose $t$-coordinates differ by a constant, is a confocal ellipse, and if the half-sum of the two $t$-coordinates is constant, then this locus is a confocal hyperbola. We refer to \cite{GIT,IT} and to a detailed discussions in \cite{St1,St2}.
\end{remark}

 \begin{remark}\label{remark:hyp}
 Note that the $T$-invariant leaf $\g$ in Proposition \ref{prop:int} need not be connected for the proposition to hold. Indeed, each level curve of $\lambda$ has two components (each topologically a  circle); in the elliptic case (level curves above and below  the $\infty$ shape in Figure \ref{space} (right)), each of these components is $T$-invariant, while in the hyperbolic case (level curves  inside the $\infty$ shape in Figure \ref{space} (right)), the two components are interchanged by $T$. By Proposition \ref{prop:int}, even in this hyperbolic case, one can put a coordinate on each of the two components, say $t$ on one component and $t_1$ on the other, such that $T$ is given by $T(t,\lambda)=(t_1+c, \lambda),$ $T(t_1,\lambda)=(t+c, \lambda),$ for some constant $c$ (depending in $\lambda$). 
 \end{remark}
 \subsection{Families of rays, envelopes, cusps}\label{sec:dual}
  The conjectures and theorems of the Introduction concern families of rays, their envelopes (or caustics) and cusps. We  briefly review here  the pertinent definitions.  See, e.g., Section 8.4 of \cite{FT}. 
 
Define the ``line'' dual to a point in $\R^2$ as 
the  curve in $\cL$ corresponding to the set  of rays incident to the point (a ``pencil" of rays). See the dotted curves in Figure \ref{fig:3pencils}. We also include ``lines" dual to ``points at infinity", corresponding  to pencils  of parallel rays sharing a common direction (vertical lines in the $(\alpha,p)$ coordinates on $\cL$). This defines a 2-parameter family of curves  in $\cL$, a unique curve  through each given point   in a  given tangent direction at this point. 
 
\begin{definition} Given a 1-parameter family of rays, that is, a smooth curve $\g\subset\cL$, an  {\em inflection} point of $\g$  is a point where the tangent ``line" to $\g$ at this point has contact of order $m\geq 2$ with $\g$ (the tangent ``line'' to a curve  has typically contact of order 1). 
\end{definition}

\begin{definition}
 The  {\em envelope} (or {\em caustic}) of $\g$ is an oriented plane curve $\Gamma$ whose set of tangent lines is $\g$. 
 \end{definition}

 \begin{definition} An $m$-cusp of a plane curve $\Gamma$ is a point for which there is a diffeomorphism taking a neighborhood of the point to a neighborhood of the origin in the $(x,y)$-plane, taking the point to $(0,0)$ and  $\Gamma$ to the curve $y^m=x^{m+1}.$ 
An {\em ordinary cusp} is a 2-cusp (or a semi-cubical cusp). 
 \end{definition}

A useful basic characterization of $m$-cusps is the following. Let $\g$ be a smooth 1-parameter family of rays with envelope $\Gamma$. Then an $m$-cusp of $\Gamma$ corresponds to an inflection point of $\g$ of order $m$. See Example 8.2 of \cite{FT}.

\section{Proof of Theorem \ref{thm:one}} \label{sec:pfone}

We restate  here  Theorem \ref{thm:one} from the Introduction.

\mn{\bf Theorem \ref{thm:one}.} {\em Let $O$ be a non-focal point inside an ellipse $C$, and let $E$ and $H$ be the confocal ellipse and hyperbola  (respectively) passing through $O$. Consider the four rays emanating from $O$ and tangent to  $E$ and $H$ (two each). Then after $n$ reflections, the 4  rays are   tangent to $E$ and $H$ at 4 points which are cusps of the $n$-th caustic by reflection from $O$. }

\bn




To prove Theorem \ref{thm:one}, we first reformulate it as a statement about the   inflection points of a  curve in the phase cyclinder $\cL$,
as explained in Section \ref{sec:dual}. 

Consider the pencil of rays incident to $O$  and let $\gamma$ be the corresponding curve  in $\cL$ (a ``line"). There are 4 points on $\gamma$, corresponding to the 4 rays tangent to the confocal conics $E$ and $H$ at $O$. The dual statement to Theorem \ref{thm:one} is then that {\em $T^n$ maps these 4 points to  inflection points of $T^n(\gamma)$}. 

We proceed as follows. 
Let $r_0\in \g$ be one of these 4 rays. We separate the proof into 3 cases (see Figure \ref{fig:3pencils}): 

\begin{enumerate}[{\bf 1.}]

\item The ray $r_0$ is one of the two rays tangent to the confocal ellipse $E$ through $O$. In this case,   $O$ may not lie on the line segment connecting to two foci. 

\item  The ray $r_0$ is one of the two rays tangent to the  confocal hyperbola $H$ through $O$. In this case, $O$ may not lie on the minor axis of $C$, nor on the major axis, on the complement of the line segment connecting the  two foci.

\item The point $O$ lies on one of the axes of $C$ and $r_0$ is one of the two rays aligned with this axis. In this case  $O$ may be  the center of $C$. 

\end{enumerate}

\begin{figure}[ht]
\centering
\def\svgwidth{.9\textwidth}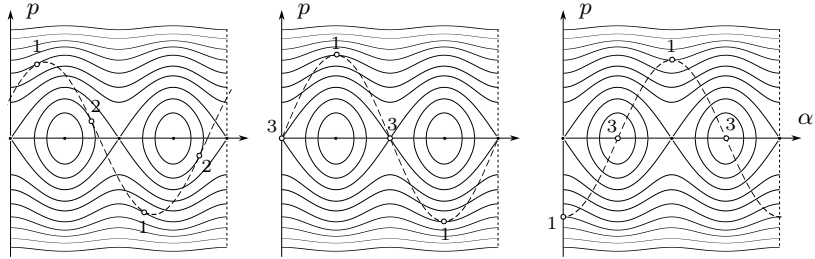
\caption{The 3 types of inflection points of pencils. In each figure, the pencil is the doted curve, with  4 points and their type marked on it. Left: $O$ does not lie on an axis of $C$ (the generic case). Middle: $O$ lies on the major axis, between a vertex and the nearby focus. Right: $O$ lies on the minor axis. }
\label{fig:3pencils}
\end{figure}

\paragraph{Case 1.} 
Let $E=C_{\lambda_0}$, $b^2<\lambda_0<a^2$,  be the confocal ellipse  passing through $O$  and $r_0\in \g$ one of the 2 rays tangent to $C_{\lambda_0}$  at $O$. The $T$-invariant curve in $\cL$ passing through $r_0$ is given by $\lambda=\lambda_0$ in the $(t,\lambda)$ coordinates. 

\mn {\em Note.} We use $r_0$ to denote both a point in $\cL$ and  the corresponding ray in $\R^2$.

\begin{lemma}\label{lemma:tang} $r_0$ is a tangency point of $\gamma$ with the $T$-invariant phase curve $\lambda=\lambda_0$.
\end{lemma}

\begin{figure}[ht]
\centering
\def\svgwidth{.3\textwidth}
\begingroup%
  \makeatletter%
  \providecommand\color[2][]{%
    \errmessage{(Inkscape) Color is used for the text in Inkscape, but the package 'color.sty' is not loaded}%
    \renewcommand\color[2][]{}%
  }%
  \providecommand\transparent[1]{%
    \errmessage{(Inkscape) Transparency is used (non-zero) for the text in Inkscape, but the package 'transparent.sty' is not loaded}%
    \renewcommand\transparent[1]{}%
  }%
  \providecommand\rotatebox[2]{#2}%
  \newcommand*\fsize{\dimexpr\f@size pt\relax}%
  \newcommand*\lineheight[1]{\fontsize{\fsize}{#1\fsize}\selectfont}%
  \ifx\svgwidth\undefined%
    \setlength{\unitlength}{634.35927984bp}%
    \ifx\svgscale\undefined%
      \relax%
    \else%
      \setlength{\unitlength}{\unitlength * \real{\svgscale}}%
    \fi%
  \else%
    \setlength{\unitlength}{\svgwidth}%
  \fi%
  \global\let\svgwidth\undefined%
  \global\let\svgscale\undefined%
  \makeatother%
  \begin{picture}(1,0.66515443)%
    \lineheight{1}%
    \setlength\tabcolsep{0pt}%
    \put(0,0){\includegraphics[width=\unitlength,page=1]{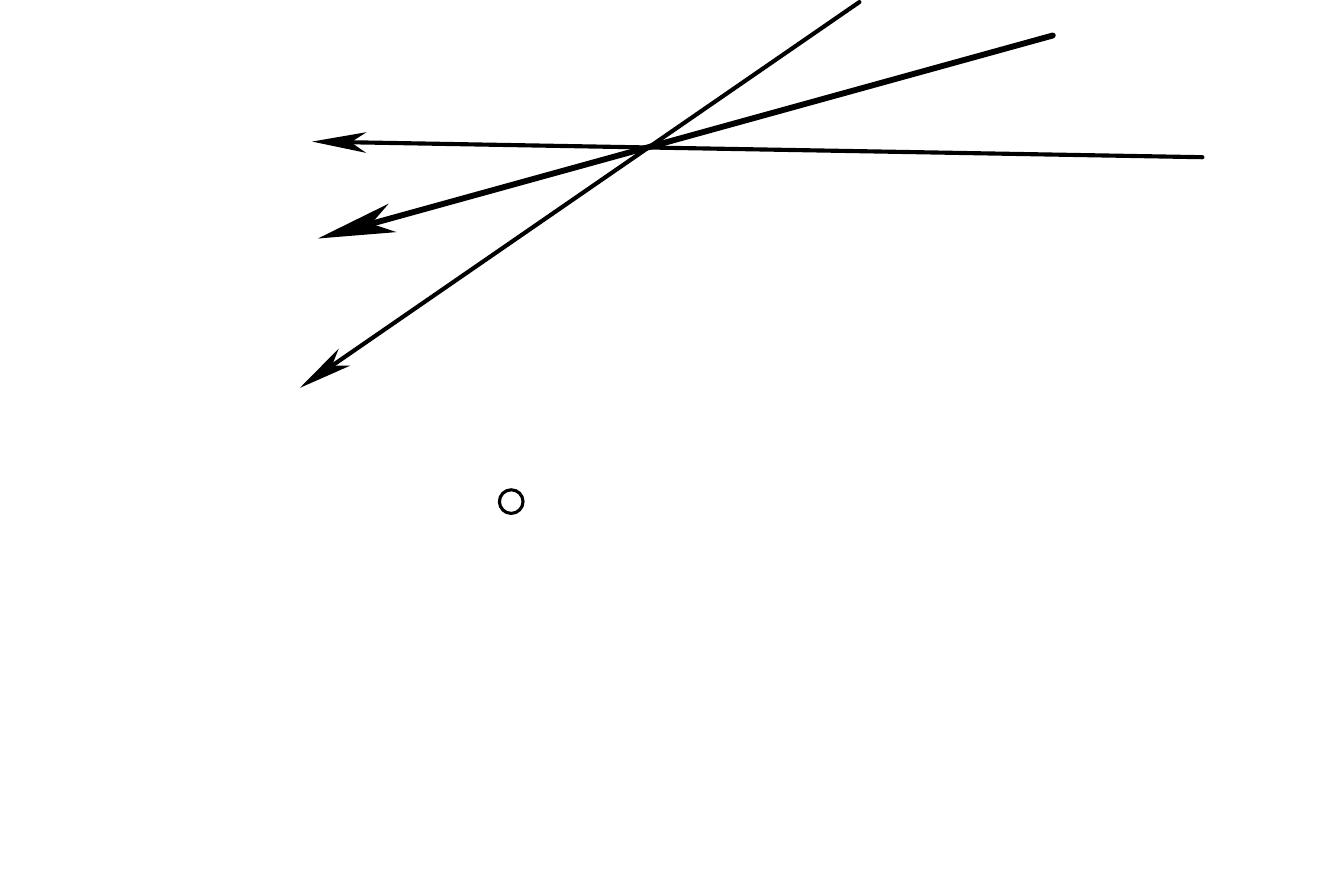}}%
    \put(0.41911074,0.12154458){\color[rgb]{0.33333333,0,0}\makebox(0,0)[lt]{\lineheight{1.25}\smash{\begin{tabular}[t]{l}\s$C_{\lambda}$\end{tabular}}}}%
    \put(0,0){\includegraphics[width=\unitlength,page=2]{pencil1.pdf}}%
    \put(0.4452804,0.594037){\color[rgb]{0.33333333,0,0}\makebox(0,0)[lt]{\lineheight{1.25}\smash{\begin{tabular}[t]{l}\s$O$\end{tabular}}}}%
    \put(0.14019117,0.46913581){\color[rgb]{0.33333333,0,0}\makebox(0,0)[lt]{\lineheight{1.25}\smash{\begin{tabular}[t]{l}\s$r_0$\end{tabular}}}}%
    \put(0,0){\includegraphics[width=\unitlength,page=3]{pencil1.pdf}}%
    \put(0.2516247,0.02110152){\color[rgb]{0.33333333,0,0}\makebox(0,0)[lt]{\lineheight{1.25}\smash{\begin{tabular}[t]{l}\s$C_{\lambda_0}$\end{tabular}}}}%
  \end{picture}%
\endgroup%

\hspace{.2\textwidth}
\def\svgwidth{.4\textwidth}
\begingroup%
  \makeatletter%
  \providecommand\color[2][]{%
    \errmessage{(Inkscape) Color is used for the text in Inkscape, but the package 'color.sty' is not loaded}%
    \renewcommand\color[2][]{}%
  }%
  \providecommand\transparent[1]{%
    \errmessage{(Inkscape) Transparency is used (non-zero) for the text in Inkscape, but the package 'transparent.sty' is not loaded}%
    \renewcommand\transparent[1]{}%
  }%
  \providecommand\rotatebox[2]{#2}%
  \newcommand*\fsize{\dimexpr\f@size pt\relax}%
  \newcommand*\lineheight[1]{\fontsize{\fsize}{#1\fsize}\selectfont}%
  \ifx\svgwidth\undefined%
    \setlength{\unitlength}{521.28331943bp}%
    \ifx\svgscale\undefined%
      \relax%
    \else%
      \setlength{\unitlength}{\unitlength * \real{\svgscale}}%
    \fi%
  \else%
    \setlength{\unitlength}{\svgwidth}%
  \fi%
  \global\let\svgwidth\undefined%
  \global\let\svgscale\undefined%
  \makeatother%
  \begin{picture}(1,0.57929606)%
    \lineheight{1}%
    \setlength\tabcolsep{0pt}%
    \put(0,0){\includegraphics[width=\unitlength,page=1]{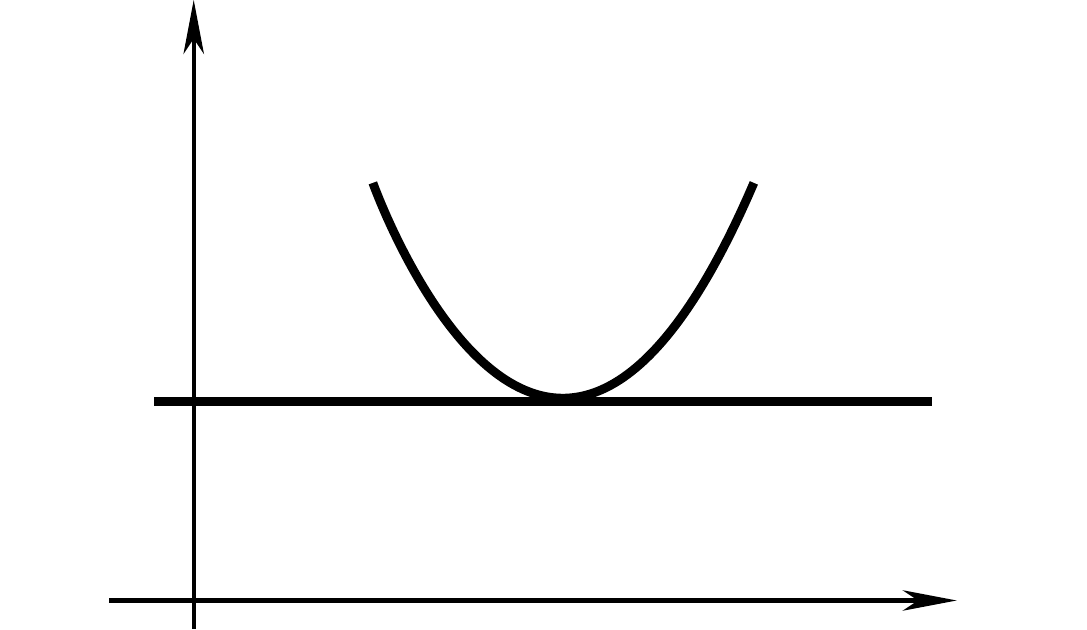}}%
    \put(0.01670553,0.28135552){\color[rgb]{0.33333333,0,0}\makebox(0,0)[lt]{\lineheight{1.25}\smash{\begin{tabular}[t]{l}\s$\lambda$\end{tabular}}}}%
    \put(0.85719845,0.06810946){\color[rgb]{0.33333333,0,0}\makebox(0,0)[lt]{\lineheight{1.25}\smash{\begin{tabular}[t]{l}\s$t$\end{tabular}}}}%
    \put(0.01316235,0.19776213){\color[rgb]{0.33333333,0,0}\makebox(0,0)[lt]{\lineheight{1.25}\smash{\begin{tabular}[t]{l}\s$\lambda_0$\end{tabular}}}}%
    \put(0.70509162,0.34976242){\color[rgb]{0.33333333,0,0}\makebox(0,0)[lt]{\lineheight{1.25}\smash{\begin{tabular}[t]{l}\s$\g$\end{tabular}}}}%
    \put(0.47751759,0.14118857){\color[rgb]{0.33333333,0,0}\makebox(0,0)[lt]{\lineheight{1.25}\smash{\begin{tabular}[t]{l}\s$r_0$\end{tabular}}}}%
    \put(0,0){\includegraphics[width=\unitlength,page=2]{pencil3.pdf}}%
  \end{picture}%
\endgroup%

\caption{Lemma \ref{lemma:tang}.}
\label{pencil}
\end{figure}

\begin{proof}

The  rays of the pencil close to $r_0$ are tangent to confocal ellipses with a greater value of the parameter $\lambda$, see Figure \ref{pencil} (left). It follows that  $\gamma$, near $r_0$, drawn in the $(t,\lambda)$  plane,  lies above the horizontal line $\lambda=\lambda_0$ and is therefore tangent to it at $r_0$. See Figure \ref{pencil} (right). \end{proof}
%

\begin{lemma}\label{lemma:inf} 
 $T(r_0)$ is an inflection point of $T(\g).$
\end{lemma}

\begin{proof}
Let $r_0=(t_0, \lambda_0)$, $r_1=(t_1,\lambda_0)=T(r_0)$, the reflection of $r_0$ by $C$, where $t_1=t_0+c(\lambda_0).$ Then $r_1$ is  tangent to $C_{\lambda_0}$ at some point, $O_1$. Let $\gamma_1$ be the ``line" dual to $O_1$, corresponding to the pencil of rays through $O_1$. 
To show that $r_1$ is an inflection point of $T(\g)$  it is then enough to show that the 2-jets at $r_1$ of $T(\g)$ and $\g_1$ coincide. See Figure \ref{pencil1}. 

\begin{figure}[ht]
\centering
\def\svgwidth{.9\textwidth}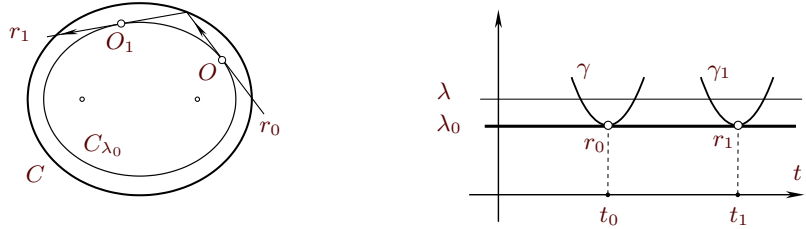
\caption{Proof of Lemma \ref{lemma:inf}.}
\label{pencil1}
\end{figure}
First, $r_1\in\g_1$, so $\g_1$ and $T(\g)$ intersect at $r_1$ (their $0$-jets coincide). 
Second, $\g$ is tangent to the $T$-invariant horizontal line  $\lambda=\lambda_0$ at $r_0$ 
(Lemma \ref{lemma:tang}) hence $T(\g)$ is tangent to $\lambda=\lambda_0$  at $r_1=T(r_0).$ 
The same holds for $\g_1$, by Lemma \ref{lemma:tang}, hence $\g_1$ and $T(\g)$ are tangent at $r_1$ 
(their $1$-jets coincide). 

Next, the curve $\gamma$ intersects the horizontal line at  a level $\lambda>\lambda_0$ 
at two points, corresponding to the rays shown in Figure \ref{pencil} (left). 
The billiard reflection in the ellipse with parameter $\lambda_0$ (the outer ellipse 
in Figure \ref{pencil} (left)) takes one of these rays to the other  one.  The difference 
of the $t$-coordinates of these two intersection points depends only on $\lambda$  and 
$\lambda_0$, but not on $t_0$ (see Remark \ref{remark:ind} of  Section \ref{sec:bil}). 
It follows that the 2-jets of $\g$ and $\g_1$, at $r_0$ and $r_1$ (respectively),  are parametrized by 
\begin{equation}\label{eq:jets}
\g: \e\mapsto (t_0+\eps,\lambda_0+a\eps^2),\quad \gamma_1: \delta \mapsto (t_1+\delta,\lambda_0+a\delta^2),
\end{equation}
where $a=a(\lambda_0)$, $t_1=t_0+c(\lambda_0).$ 

\mn{\em Note.} All calculations for the rest of the proof  of this lemma are mod $\e^3$ and $\delta^3$.
 
\mn

 Now $T(t,\lambda)=(t+c(\lambda), \lambda)$, hence the 2-jet of $T(\g)$ at $r_1$ is parametrized by
\begin{align*} \label{eq:curve1}
T(\g): \e\mapsto &(t_0+\e+c(\lambda_0+a\e^2),\lambda_0+a\eps^2)\\
&=(t_0+\eps+c(\lambda_0)+a c'(\lambda_0)\eps^2,\lambda_0+a\eps^2)\\
&=(t_1+\eps+a c'(\lambda_0)\eps^2,\lambda_0+a\eps^2).
\end{align*}
 Next  we reparametrize this 2-jet by setting 
 $$\delta=\eps+a c'(\lambda_0)\eps^2,
 $$
 with  inverse (mod $\delta^3$), 
 $$\e=\delta-a c'(\lambda_0)\delta^2.
 $$
 It follows that the 2-jet of $T(\g)$ at $r_1$ can be reparametrized as 
 $$T(\g): \delta\mapsto (t_1+\delta, \lambda_0+a\delta^2),
 $$
 coinciding with the expression \eqref{eq:jets}  for the 2-jet of $\gamma_1$ at $r_1$, as needed. 
\end{proof}

Note that  the last two lemmas are statements about the 2-jet of $\g$ at $r_0$. That is, they remain valid  if one replaces $\g$ with a curve whose 2-jet at $r_0$ coincides with that of $\g$. We thus conclude: if $r_0$ is an inflection point of a curve  $\g\subset\cL$, which is also a point of tangency of $\g$ with the  leaf of the $T$-invariant foliation of $\cL$ dual to an ellipse $E$ confocal to $C$, then the same holds for $T(r_0)\in T(\g).$ It follows by induction on $n$ that the same holds for $T^n(r_0)\in T^n(\g)$. This  proves Case {\bf 1} of Theorem \ref{thm:one}.


\paragraph{Case 2.} This case is very similar to the previous one, so we omit  the details. We only note that in this case, like in case {\bf 1},  the $T$-invariant leaf $\lambda=\lambda_0$ consists of two components, but unlike case {\bf 1},  $T^n$, for $n$ odd, interchanges the two components; the argument however is unaffected. See Remark 
\ref{remark:hyp} and Figure \ref{fig:hyp}. 

\begin{figure}[ht]
\centering
\def\svgwidth{.9\textwidth}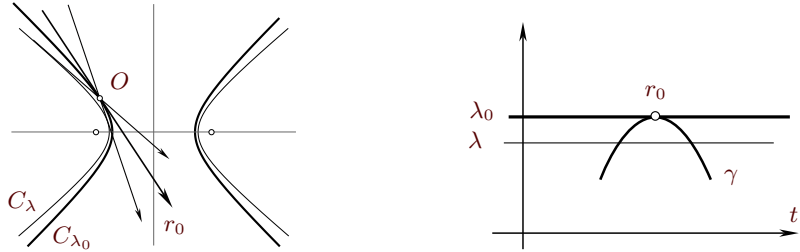
\caption{Case {\bf 2} of  Lemma \ref{lemma:inf}. Compare to Figure  \ref{pencil}.}
\label{fig:hyp}
\end{figure}

\paragraph{Case 3.} This case is simpler then  the previous two. First,  a lemma. 

\begin{lemma}\label{lemma:rho} Let $\rho$ denote the involution  of $\cL$ induced by the reflection about one of the axes of $C$, major or minor. Let $r_0$ be one of the  two fixed points of $\rho$ (a ray aligned with the axis of reflection) and  $\g\subset \cL$ a $\rho$-invariant curve containing $r_0$.  Then $r_0$ is an inflection point of $\g$. 
\end{lemma}
\begin{proof} Assume that $\rho$ is given by reflection about the major axis of $C$ (the $x$-axis) and $r_0$ is the ray along this axis, oriented eastwards. We use the coordinates $(\alpha, p)$ on $\cL$, see Figure \ref{fig:ap}. Then $\rho(\alpha, p)=(-\alpha, -p)$ and $r_0=(0,0).$ Assume the tangent to $\g$ at $r_0$ is not vertical. Then the 2-jet of $\g$   at  $r_0$ can be parametrized by  
\be\label{eq:eta1}
\e\mapsto (\e, a\e + b\e^2),
\ee
 for some $a,b\in \R.$ This is mapped by $\rho$ to 
 $$\e\mapsto
(-\e, -a\e - b\e^2).$$
 Renaming $-\e$ by $\e$, this 2-jet of $\rho(\g)$ at $r_0$ can be reparametrized as 
 \be\label{eq:eta2}
 \e\mapsto (\e, a\e -b\e^2).
 \ee
 Since $\rho(\g)=\g$ and $r_0$ is a fixed point of $\rho$, the 2-jets  \eqref{eq:eta1} and \eqref{eq:eta2} must coincide. 
 It follows that $b=0$, hence the 2-jet of $\g$ at $r_0$ is parametrized by 
 \be\label{eq:eta3}
 \e\mapsto (\e, a\e).
 \ee 
 
 On the other hand,   the tangent ``line" to $\g$ at $r_0$ is the graph of $p=a\sin\alpha$  (see Lemma \ref{lemma:pencil} below). Its 2-jet at $r_0$ is given by  formula \eqref{eq:eta3}.  This shows  that $r_0$ is an inflection point of $\g$. 
 
 If the tangent to $\g$ at $r_0$ is vertical then the  tangent ``line" at $r_0$ is 
 $\alpha=0$ and we can parametrize the 2-jet of $\g$ at $r_0$ by 
 $\e\mapsto (a\e^2, \e)$. As before, $\rho$-invariance of $\g$ implies that $a=0,$ hence the 2-jets of $\g$ 
 and the $\alpha=0$ at $r_0$ coincide. Thus in this case $r_0$ is an inflection point of $\g$ as well. 
 
 The other 3 cases, where  $r_0=(\pi, 0)$ and $\rho$ is the reflection about the $x$-axis, or $\rho$ is the reflection about the $y$-axis and $r_0=(\pm\pi/2, 0)$,  are treated similarly  and their proof is omitted. \end{proof}

We can now complete the proof of Case {\bf 3} of Theorem \ref{thm:one}. Let $O$ be a point on one of the axes of $C$ (major or minor, or both, when $O$ is the center of $C$, if $C$ is not a circle). Let $\g\subset \cL$ 
be the dual ``line" (the curve corresponding to the pencil of rays through $O$). 
Let $r_0\in \g$ be one of the two rays aligned with the axis through $O$. 
Then $\g$ is $\rho$-invariant and $r_0$ is a fixed point of $\rho$. Clearly, $\rho$ and $T$ commute, hence 
$T^n(\g)$ is $\rho$-invariant and $T^n(r_0)\in  T^n(\g)$ is a fixed point of $\rho$. 
Lemma \ref{lemma:rho} implies that $T^n(r_0)$ is an inflection point of $T^n(\g),$ as needed. \qed

\section{Proof of Theorem \ref{thm:two}} \label{sec:pftwo}
\subsection{Two lemmas}
The billiard table $C$ here is the unit circle $x^2+y^2=1$. We use the same coordinates $(\alpha,p)$ in the space of oriented lines in $\R^2$ that were introduced in Section \ref{sec:bil}, Figure \ref{fig:ap}. 


\begin{lemma}\label{lemma:pencil}
The pencil of rays through a point $(a,b)\in\R^2$, the  ``line" dual to $(a,b)$, is given by the equation 
\begin{equation} \label{eq:line}
p(\alpha)=a\sin\alpha-b\cos\alpha.
\end{equation}
See Figure \ref{fig:line}. 
\end{lemma}

\begin{figure}[h]
\centering
\def\svgwidth{.45\textwidth}
\begingroup%
  \makeatletter%
  \providecommand\color[2][]{%
    \errmessage{(Inkscape) Color is used for the text in Inkscape, but the package 'color.sty' is not loaded}%
    \renewcommand\color[2][]{}%
  }%
  \providecommand\transparent[1]{%
    \errmessage{(Inkscape) Transparency is used (non-zero) for the text in Inkscape, but the package 'transparent.sty' is not loaded}%
    \renewcommand\transparent[1]{}%
  }%
  \providecommand\rotatebox[2]{#2}%
  \newcommand*\fsize{\dimexpr\f@size pt\relax}%
  \newcommand*\lineheight[1]{\fontsize{\fsize}{#1\fsize}\selectfont}%
  \ifx\svgwidth\undefined%
    \setlength{\unitlength}{601.52888489bp}%
    \ifx\svgscale\undefined%
      \relax%
    \else%
      \setlength{\unitlength}{\unitlength * \real{\svgscale}}%
    \fi%
  \else%
    \setlength{\unitlength}{\svgwidth}%
  \fi%
  \global\let\svgwidth\undefined%
  \global\let\svgscale\undefined%
  \makeatother%
  \begin{picture}(1,0.5635307)%
    \lineheight{1}%
    \setlength\tabcolsep{0pt}%
    \put(0,0){\includegraphics[width=\unitlength,page=1]{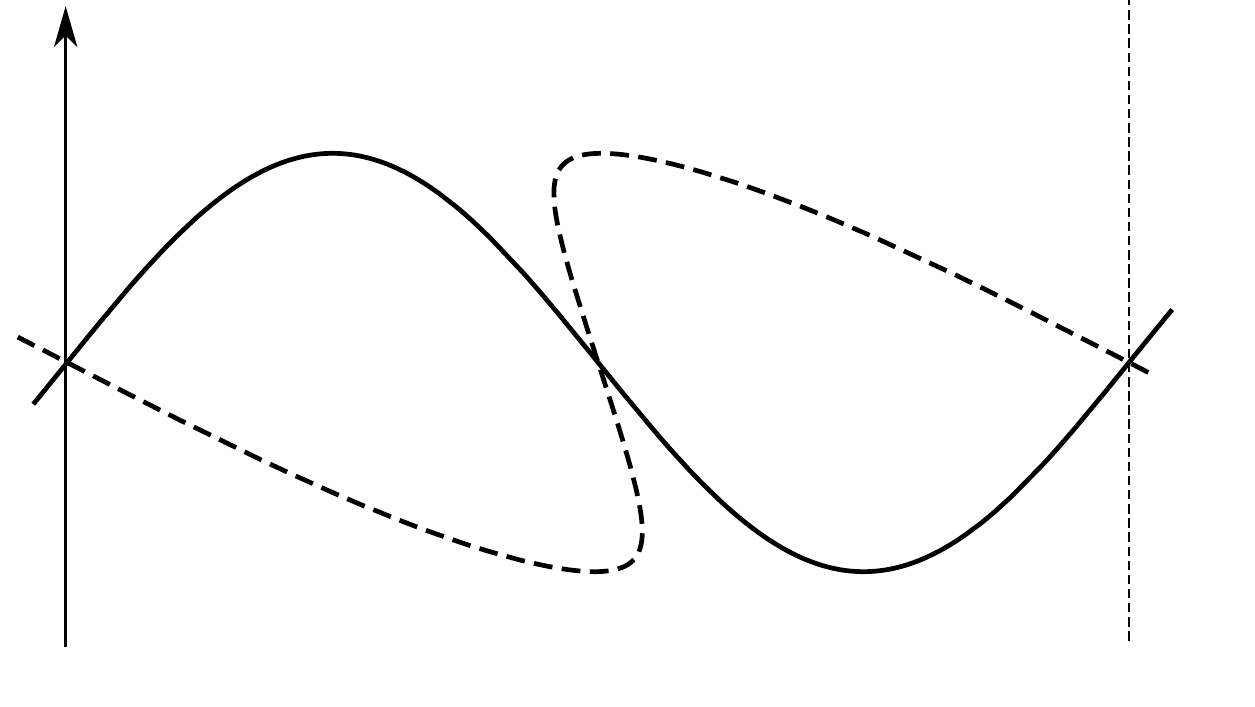}}%
    \put(0.07948526,0.53691878){\color[rgb]{0.02352941,0.03137255,0.02352941}\makebox(0,0)[lt]{\lineheight{1.25}\smash{\begin{tabular}[t]{l}\s$p$\end{tabular}}}}%
    \put(0.97330683,0.29951564){\color[rgb]{0.02352941,0.03137255,0.02352941}\makebox(0,0)[lt]{\lineheight{1.25}\smash{\begin{tabular}[t]{l}\s$\alpha$\end{tabular}}}}%
    \put(0,0){\includegraphics[width=\unitlength,page=2]{caustic_circle3.pdf}}%
  \end{picture}%
\endgroup%

\hspace{.18\textwidth}
\def\svgwidth{.25\textwidth}
\begingroup%
  \makeatletter%
  \providecommand\color[2][]{%
    \errmessage{(Inkscape) Color is used for the text in Inkscape, but the package 'color.sty' is not loaded}%
    \renewcommand\color[2][]{}%
  }%
  \providecommand\transparent[1]{%
    \errmessage{(Inkscape) Transparency is used (non-zero) for the text in Inkscape, but the package 'transparent.sty' is not loaded}%
    \renewcommand\transparent[1]{}%
  }%
  \providecommand\rotatebox[2]{#2}%
  \newcommand*\fsize{\dimexpr\f@size pt\relax}%
  \newcommand*\lineheight[1]{\fontsize{\fsize}{#1\fsize}\selectfont}%
  \ifx\svgwidth\undefined%
    \setlength{\unitlength}{202.05368042bp}%
    \ifx\svgscale\undefined%
      \relax%
    \else%
      \setlength{\unitlength}{\unitlength * \real{\svgscale}}%
    \fi%
  \else%
    \setlength{\unitlength}{\svgwidth}%
  \fi%
  \global\let\svgwidth\undefined%
  \global\let\svgscale\undefined%
  \makeatother%
  \begin{picture}(1,0.73313969)%
    \lineheight{1}%
    \setlength\tabcolsep{0pt}%
    \put(0,0){\includegraphics[width=\unitlength,page=1]{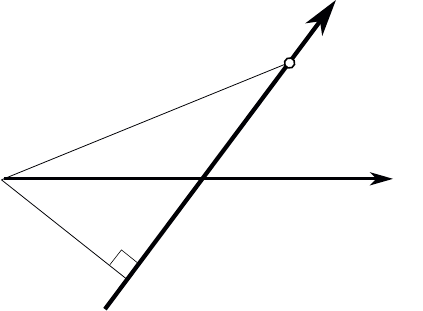}}%
    \put(0.58123122,0.35096549){\makebox(0,0)[lt]{\lineheight{1.25}\smash{\begin{tabular}[t]{l}\s$\alpha$\end{tabular}}}}%
    \put(0.72470229,0.5352533){\makebox(0,0)[lt]{\lineheight{1.25}\smash{\begin{tabular}[t]{l}\s$(a,b)$\end{tabular}}}}%
    \put(0,0){\includegraphics[width=\unitlength,page=2]{palpha1.pdf}}%
    \put(0.07389981,0.10088517){\makebox(0,0)[lt]{\lineheight{1.25}\smash{\begin{tabular}[t]{l}\s$p$\end{tabular}}}}%
    \put(0.24839226,0.33836447){\makebox(0,0)[lt]{\lineheight{1.25}\smash{\begin{tabular}[t]{l}\s$\theta$\end{tabular}}}}%
    \put(0,0){\includegraphics[width=\unitlength,page=3]{palpha1.pdf}}%
    \put(0.32105259,0.49056534){\makebox(0,0)[lt]{\lineheight{1.25}\smash{\begin{tabular}[t]{l}\s$r$\end{tabular}}}}%
    \put(0.16149444,0.21222569){\makebox(0,0)[lt]{\lineheight{1.25}\smash{\begin{tabular}[t]{l}\s$\alpha'$\end{tabular}}}}%
  \end{picture}%
\endgroup%

\caption{ Left: the solid curve represents the pencil of rays through a point inside a circular table $C$.  The dotted curve is its image under the billiard map $T$. The 4 marked points on it are its inflection points. The horizontal lines are the $T$ invariant foliation of the phase cylinder $\cL$. Right: the proof of Lemma \ref{lemma:pencil}.}\label{fig:line}
\end{figure}

\begin{proof} 
Let $(a,b)=r(\cos\theta, \sin\theta)$ and $\alpha'=\pi/2-\alpha.$ Then
\begin{align*}
p&=r\cos(\theta+\alpha')=r(\cos\theta \cos\alpha'-\sin\theta\sin\alpha')\\
&=r(\cos\theta\sin\alpha-\sin\theta\cos\alpha)=a\sin\alpha-b\cos\alpha.
\qedhere
\end{align*}
\note{}
\end{proof}


Let $\g$ be a curve in the phase space ${\mathcal L}.$ Using the same terminology as in Section \ref{sec:pfone}, 
an inflection point of $\g$ is a  second order tangency with the ``line" tangent to $\g$ at the point. 
If $\g$ is the  graph  of a function $p(\alpha)$, the tangent ``line" is a  graph of a function given by  (\ref{eq:line}), i.e.,  a solution to the ODE $f''+f=0$, hence  the  inflection points of $\g$ are given by  the zeros of the function $p''(\alpha)+p(\alpha)$.

If a  line tangent to $\gamma$ is vertical, i.e.,  $\gamma$ is tangent at $r_0=(\alpha_0, p_0)$ to the vertical line $\alpha=\alpha_0$, then $\g$ is the graph of a function $\alpha(p)$ near $r_0$, and $r_0$ is an inflection point \iff $\alpha(p)=\alpha_0+O(|p-p_0|^3),$ degenerate if $\alpha(p)=\alpha_0+O(|p-p_0|^4).$ 

\mn 

Next consider  a map $T:\cL\to\cL$ given by 
$$
T(\alpha,p)= (\tilde\alpha,p), \ \tilde\alpha=\alpha+\phi(p)\ ({\rm mod}\ 2\pi), 
$$ 
where $\phi(p)$ is some function. 
Let $(\alpha_0, p_0)$ be the coordinates of a point $r_0$ on a curve $\g\subset\cL,$ the graph of a function $p(\alpha)$. 
We ask: what is the condition on the 2nd order jets of  $p(\alpha)$ and 
$\phi(p)$ at  $\alpha_0$ and  $p_0$ (respectively)  so that  $T(\g)$ has an inflection point at $T(r_0)$? The answer is given by the following lemma.

\begin{lemma} \label{lemma:inf3}
Let $\gamma$ be the graph of $p(\alpha)$, $r_0=(\alpha_0, p_0)\in\gamma$, 
with
\begin{align*}
&p(\alpha_0+\e)=p_0+p_1\e+{p_2\over 2}\e^2+O(\e^3),\\
& \phi(p_0+\delta)=\phi_0+\phi_1\delta+{\phi_2\over 2}\delta^2+O(\delta^3).
\end{align*}
Then  $T(r_0)$ is an inflection point of  $T(\g)$ if and only if 
\begin{equation} \label{eq:finf} 
p_2+ p_0(1+p_1 \phi_1)^3=p_1^3\phi_2.
\end{equation}
\end{lemma}

\begin{proof}Calculating mod $\e^3, \delta^3$ throughout, set 
$$\delta= p_1\e+{p_2\over 2}\e^2,$$
then 
\begin{align*}
\phi(p(\alpha_0+\e))&=\phi(p_0+\delta)=\phi_0+\phi_1\delta+{\phi_2\over 2}\delta^2\\
&=\phi_0+\phi_1 p_1\e+{p_1^2\phi_2+p_2\phi_1\over 2}\e^2.
\end{align*}
The 2-jet of $\g$ at  $r_0=(\alpha_0, p_0)$ is parametrized by
$$\e\mapsto\left(\alpha_0+\e,\  p_0+p_1\e+{p_2\over 2}\e^2\right), 
$$
hence the 2-jet of $T(\g) $ at $T(r_0)=(\alpha_0+\phi_0, p_0)$ is parametrized by 
$$\e\mapsto\left(\alpha_0+\phi_0+(1+p_1\phi_1 )\e+{p_1^2\phi_2+p_2\phi_1\over 2}\e^2,\ p_0+p_1\e+{p_2\over 2}\e^2\right).
$$
Let 
$$\tilde\e:=(1+p_1\phi_1 )\e+{p_1^2\phi_2+p_2\phi_1\over 2}\e^2,
$$
then, assuming $1+p_1\phi_1\neq 0$, one can invert this (mod $\tilde\e^3$), 
$$\e=\frac{\tilde\e}{1+p_1 \phi_1}-\frac{p_1^2 \phi_2+p_2 \phi_1}{2 (1+p_1 \phi_1)^3}\tilde\e^2.
$$
Thus the  2-jet of $T(\g) $ at $T(r_0)$ is parametrized by 
$$
\tilde\e\mapsto\left(\alpha_0+\phi_0+\tilde\e,\ 
p_0+\tilde p_1 \tilde\e+ {\tilde p_2\over 2}\tilde\e^2\right), 
$$
where
$$\tilde p_1=\frac{p_1}{1+p_1\phi_1}, \ \tilde p_2= \frac{p_2-p_1^3 \phi_2}{(1+p_1 \phi_1)^3}.$$
The inflection condition at $r_1$ is then $\tilde p_2+p_0=0,$ which reduces to the stated formula (\ref{eq:finf}).

If $1+p_1\phi_1=0$ then $p_1=p'(\alpha_0)\neq 0$ so one can invert $p(\alpha)$ near $\alpha_0$, 
$$\alpha(p_0+\delta)=\alpha_0+\alpha_1\delta+{\alpha_2\over 2}\delta^2, 
$$
where
\be
\label{eq:inv1}
\alpha_1={1\over p_1}, \ \alpha_2=-\frac{p_2}{p_1^3}
\ee
and 
\be
\label{eq:inv2}
p_1={1\over \alpha_1}, \ p_2=-\frac{\alpha_2}{\alpha_1^3}.
\ee 
 
The inflection condition for $p(\alpha)$ at $\alpha_0$  is $p_2+p_0=0.$ Substituting for $p_2$ from equation  \eqref{eq:inv2}, this is 
\begin{equation} \label{eq:gp}
\alpha_2=p_0(\alpha_1)^3.
\end{equation}
Now $T(\g)$ is the graph of $\alpha(p)+\phi(p),$ hence the inflection condition at $T(r_0)$  is
$$\alpha_2+\phi_2=p_0(\alpha_1+\phi_1)^3.
$$
Substituting for $\alpha_1,\alpha_2$ from equation (\ref{eq:inv1}), one obtains equation  (\ref{eq:finf}). 
\end{proof}

\subsection{Cusps by reflection in a circle}
 The billiard ball map inside the unit circle $C$ is given by $T(\alpha,p)=(\alpha+2\arccos p,p)$. Fix a point $O=(a,b)$ inside $C$ and let $\g$ be the dual ``line"  (\ref{eq:line}). 
This takes us to the setting of Lemma \ref{lemma:inf} with 
$$
p(\alpha)=a\sin\alpha-b\cos\alpha,\ \phi(p)=2n\arccos(p), \ -1<p<1.
$$

We are looking for points $r_0=(\alpha_0, p_0)\in \g$ such that  $T^n(r_0)$ is an inflection point of $T(\g)$. Using circular symmetry, we may assume, without loss of generality, that  $\alpha_0=0$, $0\leq b< 1$ and $0<a^2+b^2<1.$ 
We substitute in  formula \eqref{eq:finf}
$$p_1=a,\ p_2=-p_0=b,\ \phi_1={-2n\over\sqrt{1-b^2}},\
\phi_2={2bn\over  (1-b^2)^{3/2}},$$ 
obtaining 
the inflection condition at $T^n(r_0)$:
\be\label{eq:abn}
b-b\left[1-{2an\over \sqrt{1-b^2}}\right]^3={2a^3bn\over (1-b^2)^{3/2}}.
\ee
This is satisfied if $a=0$ or $b=0$, corresponding to four inflection points of the curve $T^n(\g)$, as described by Theorem \ref{thm:two}.

We claim that there  are no other solutions to equation \eqref{eq:abn} with  $n\geq 1$ and $0<a^2+b^2<1$.  
Set $x=a/\sqrt{1-b^2}$. Assuming $a,b\neq 0$,  equation \eqref{eq:abn} becomes
\be\label{eq:quad}
(4 n^2 -1)x^2- 6  n x+ 3 =0.
\ee
The  discriminant of this quadratic equation  in $x$ is a positive multiple of $1-n^2$. Thus equation \eqref{eq:quad} has a solution with $n\geq 1$ only for $n=1$. But in this case the solution is $x=1$, that is, $a^2+b^2=1,$ which is out of range.\qed

\subsection{The four cusps are ordinary} Dually, this amounts to proving the non-degeneracy of the 4  inflection points of $T^n(\g)$. Suppose, without loss of generality,  that $O=(a,0)$, $a>0$, hence $\g$ is given by $p=a\cos\alpha$, and  the inflection points of $T^n(\gamma)$ are $T^n(r_0)$, where 
$r_0=(0,0),(\pi,0)$ or $ \pm (\pi/2, a).$ 

Begin with $r_0=(0,0).$ The 3-jet of $\g$ at this point is parametrized by 
$$\e\mapsto (\e, a\e-{a\over 6}\e^3).$$
Then $r_n:=T^n(r_0)=(n\pi,0)$. We calculate mod $\e^4$:
$$\arccos(a\e-{a\over 6}\e^3)=\frac{\pi }{2}-a \e +\frac{a(1-a^2)}{6} \e ^3, 
$$
hence  the 3-jet of $T^n(\gamma)$ at $r_n$ is parametrized by 
\be\label{eq:jinf}\e\mapsto \left(n\pi +(1-2na) \e +\frac{na\left(1-a^2\right) }{3} \e ^3, a\e-{a\over 6}\e^3\right).
\ee
Let
$$\tilde\e:=(1-2na) \e +\frac{na\left(1-a^2\right)}{3}  \e ^3.$$
If $1-2na\neq 0$ this can be inverted, 
$$\e=\frac{ \tilde\e}{1-2 na}-\frac{ n a \left(1-a^2\right) \tilde\e^3}{3 (1-2 n a)^4},
$$
so that 
$$a\e-{a\over 6}\e^3=\frac{a }{1-2 na }\tilde\e-\frac{a\left(1-2 n a^3\right) }{6 (1-2 na )^4} \tilde\e^3.$$
The 3-jet of $T^n(\gamma)$ at $r_n$ can thus be reparametrized as
$$\tilde\e\mapsto \left(n\pi+\tilde\e, \frac{a }{1-2 na }\tilde\e-\frac{a\left(1-2 na^3\right) }{6 (1-2 na )^4}\tilde\e^3 \right). 
$$
The tangent ``line" at $r_n=(n\pi,0)$ is the graph of 
$$p(\alpha)= \frac{a }{1-2 na } \sin(\alpha-n\pi),$$
with 3-jet at $r_n$ parametrized by 
$$\e\mapsto\left(n\pi+\e, \frac{a}{1-2 na }\e -\frac{a}{6(1-2 na) }\e^3 \right).$$
This coincides with the 3-jet of $T^n(\gamma)$ at $r_n$ if and only if
$$\frac{a\left(1-2 na^3\right)  }{ (1-2 na )^4}=\frac{a }{1-2 na},
$$
which simplifies to
\be \label{eq:quad1}
(4n^2-1)a^2-6na+3=0.
\ee
The only solution is $a=n=1$, which is excluded. 

\begin{remark} We notice  a mysterious coincidence between  Equations \eqref{eq:quad} and \eqref{eq:quad1}.  We could not find an explanation. 
\end{remark}

If $1-2na=0$ then the parametrized 3-jet \eqref{eq:jinf} becomes
\be\label{eq:jinf1}\e\mapsto \left(n\pi  +\frac{1-a^2 }{6} \e ^3, a\e-{a\over 6}\e^3\right).
\ee
Let $$\tilde\e:=a\e-{a\over 6}\e^3, $$ with inverse
$$\e=\frac{\tilde\e}{a}+\frac{\tilde\e^3}{6 a^3}. 
$$
Then \eqref{eq:jinf1} can be reparametrized as 
$$
\tilde\e\mapsto \left( n\pi+\frac{\left(1-a^2\right) }{6 a^3}\tilde\e^3, \tilde\e\right).
$$
This is vertical at $r_n=(n\pi,0)$, so the tangent ``line" at $r_n$ is the vertical line $\alpha=n\pi$. It coincides with the 2-jet of the above, but not with the 3-jet, as claimed. 
The argument for $r_0=(\pi,0)$ is similar and is omitted. 

\mn

For $r_0=(\pi/2, a)$ we proceed in a similar way. The ``line" $\g$ is the graph of $p=a\sin\alpha$, whose 3-jet at $r_0$ is parametrized by 
$$\e\mapsto \left({\pi\over 2}+\e, a-{a\over 2}\e^2\right). 
$$
The image of this 3-jet under $T^n$ is the 3-jet at $T^n(r_0)$ parametrized by

$$\e\mapsto \left(\alpha_n+\e+\frac{a n }{\sqrt{1-a^2}}\e^2, a-{a\over 2}\e^2\right), \quad \alpha_n={\pi\over 2}+2n\arccos a. 
$$
Let 
$$\tilde\e:=\e+\frac{a n}{\sqrt{1-a^2}}\e^2, 
$$
with inverse 
$$\e=\tilde\e-\frac{a n}{\sqrt{1-a^2}} \tilde\e^2+\frac{2 a^2 n^2 }{1-a^2}\tilde\e^3.
$$
We get the parametrization of the 3-jet of $T^n(\g)$ at $T^n(r_0),$
\be\label{eq:jet6}\tilde\e\mapsto \left(\alpha_n+\tilde\e, a-{a\over 2} \tilde\e^2+\frac{ a^2 n}{\sqrt{1-a^2}} \tilde\e^3\right). 
\ee
The  ``line" tangent to $T^n(\g)$ at $T^n(r_0)$ is given by $p=a\cos(\alpha-\alpha_n)$, with 3-jet at  $T^n(r_0)$ parametrized by 
$$\e\mapsto \left(\alpha_n+\e, a-{a\over 2}\e^2\right). 
$$
This coincides with the 2-jet of \eqref{eq:jet6}, but not the 3rd, as claimed. The case $r_0=(-\pi/2,-a)$ is similar and is omitted. \qed

\section{Miscellanea} \label{sec:misc} 

We present here briefly some results and conjectures, inspired by the previous sections. 

\subsection{Liouville billiards} \label{sec:Lio}

Recall that a Riemannian metric in a 2-dimensional domain is called a \emph{Liouville metric} if there exist coordinates $(x,y)$ in which it is given  by the formula
$$
(f(x)+g(y)) (dx^2+dy^2),
$$
where $f$ and $g$ are smooth 
functions of one variable, such that $f(x)+g(y)>0$ for all $x,y$. The coordinate lines form a  \emph{Liouville net},  consisting of two families of mutually orthogonal curves. 

The Euclidean metric in the plane admits a Liouville net consisting of confocal conics, corresponding to the respective elliptic coordinates.  The degenerations of this net include the net of confocal parabolas and the net consisting of concentric circles and the radial lines (as well as the trivial net consisting of the horizontal and vertical lines). 

The elliptic coordinates in 3-space, restricted to a triaxial ellipsoid which is a level surface of one of the coordinates, define a Liouville metric whose Liouville net consists of the lines of curvature, see Figure \ref{ellipsoid}.

\begin{figure}[ht]
\centering
\includegraphics[width=.4\textwidth]{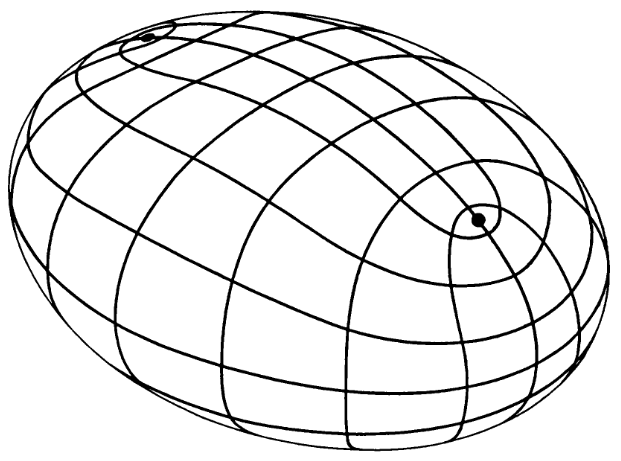}
\caption{
The lines of curvature on an ellipsoid form a Liouville net, associated with the elliptic coordinates in $\R^3$. The billiard  bounded by such a curve is a completely integrable.}
\label{ellipsoid}
\end{figure}

One considers a billiard system  in a geodesically convex domain with a smooth closed boundary on a Riemannian surface: the trajectories are made of geodesic segments, and the law of reflection is the same as in the Euclidean case (the angle of incidence equals the angle of reflection). Similar to the case of billiards in an ellipse in the plane, the billiard system on a Liouville surface whose billiard table is bounded by a coordinate line  from the  Liouville net is integrable: a generic trajectory has all its segments tangent to a fixed curve of the  Liouville net. See \cite{GIT,IT,PT1,PT2,PT3} for details.

The main ingredient in  the proof of Theorem \ref{thm:one} was the complete integrability of the billiard ball map in ellipses and its consequences, such as a version of the Arnold-Liouville theorem (Proposition \ref{prop:int}). For this reason, 
Theorem \ref{thm:one} and its proof extend, with appropriate adjustments, to Liouville billiards as well. 

We note that this set-up includes billiards bounded by conics in the hyperbolic and spherical geometries, the closest ``relatives" of the Euclidean billiard inside an ellipse. Concerning spherical and hyperbolic conics, see, e.g., \cite{Izm}. 

\subsection{Cusps on axes}\label{sec:axes} 
As noted in Remark \ref*{remarks:thm1}\ref{it:axes}, when a  light source $O$ is placed on one of the axes of an ellipse, two of the cusps on the $n$-th caustic by reflection from $O$ will be located on this axis, but Theorem \ref{thm:one} does not give their location. Here we fill this gap,  using the classical ``mirror equation" of geometric optics (Equation (5.9) of \cite{Ta}).

\begin{prop}
 Let  $O=(x_0, 0),$ $|x_0|<a,$ and let $O_n$ (resp. $O_n'$) be the cusp of the $n$-th caustic by reflection from $O$ along the trajectory leaving $O$ in the positive (resp. negative) direction of the $x$-axis.  Then  
$$O_n=(-1)^n f^n(O), \quad O'_n=(-1)^{n+1} f^n(-O), 
$$ 
where $f$ is a hyperbolic M\"obius transformation of the $x$-axis with fixed points at  the foci  $\pm F=(\pm c,0)$, $c=\sqrt{a^2-b^2}$. Furthermore, $F$ is an unstable fixed point of $f$ and  $-F$ is stable. Thus, as $n\to\infty$, 
$$O_{2n}\to -F, \quad O_{2n+1}\to F, \quad O'_{2n}\to F, \quad O'_{2n+1}\to -F.
$$

 Explicitly, 
\be\label{eq:mob}f(x)=\frac{( a^2+c^2) x-2 a c^2}{-2 a x+ a^2+c^2}. 
\ee
Exception: if $C$ is a circle then $f$ is parabolic, with a single fixed point at $(0,0)$. Thus,  $\lim O_n=\lim O_n'=(0,0)$,  as $n\to\infty.$
\end{prop}
\begin{proof}Let $(R(x),0)$ be the image of $(x,0)$ after reflection off $C$ at $(a,0)$ and  $(L(x),0)$ the image after reflection at $(-a,0)$. The $x$-coordinate of the successive images of $(x_0,0)$, starting with a reflection at $(a,0)$,  are then
 $$R(x_0), LR(x_0), RLR(x_0)\ldots.$$
Note that $L(x)=-R(-x),$ hence the $n$-th term in the above sequence  is 
$$x_n=(-1)^n(-R)^n(x_0).
$$ 
It remains to find an explicit formula for $f(x):=-R(x)$. 

The  ``mirror equation" states that if an object is placed on the line normal to a convex mirror, where the curvature of the mirror is $k$, at a distance $d$ from the mirror, then a reflected  image of the object  will form at a distance $d'$ from the mirror, given by 
\be\label{eq:mirror}
{1\over d}+{1\over d'}={2k}.
\ee
The curvature of $C$ at $(a, 0)$ is $a/b^2$, so setting $d=a-x, d'=a+f(x)$ in formula \eqref{eq:mirror},  we obtain
\be\label{eq:RL} {1\over a-x}+{1\over a+f(x)}={2a\over b^2}. 
\ee
Formula \eqref{eq:mob} for $f(x)$ follows. From formula \eqref{eq:mob} follows that 
$$f'(c)=\frac{(a+c)^2}{(a-c)^2}, \quad f'(-c)=\frac{(a-c)^2}{(a+c)^2}.
$$
Thus $f'(c)>1$ and $0<f'(-c)<1.$ It follows that $c$ is an unstable fixed point of $f$ and $-c$ is stable.

The formula for $O'_n$ is obtained  is a similar manner by considering the sequence 
 $L(x_0), RL(x_0), LRL(x_0),\ldots$.
\end{proof}

Next we study  the case when  $O$ is on the minor axis. 

\begin{prop}
 Let  $O=(0, y_0),$ $|y_0|<b,$ and let $O_n$ (resp. $O_n'$) be the cusp of the $n$-th caustic by reflection from $O$ along the trajectory leaving $O$ in the positive (resp. negative) direction of the $y$-axis.  Then  
$$O_n=(-1)^n g^n(O), \quad O'_n=(-1)^{n+1} g^n(-O), 
$$ 
where $g$ is an elliptic  M\"obius transformation of the $y$-axis, conjugate to a rotation by  $4\theta$, where $c+ib=ae^{i\theta}$ (that is, $\theta$ is  the angle  between the $x$-axis and line through $(0,b)$ and $ -F=( -c, 0)$). 

 Explicitly, 
\be\label{eq:mmob}g(y)=\frac{y \left(c^2-b^2\right)-2 b c^2}{2 b y+c^2-b^2}. 
\ee
\end{prop}
\begin{proof} 
The proof of formula \eqref{eq:mmob} is very similar to the above proof of formula \eqref{eq:mob}  and  is omitted. Using formula \eqref{eq:mmob},  one finds that  $g(y)$ has no fixed points, hence it is elliptic,  i.e.,  conjugate to a rotation. The angle of rotation is given by the derivative  at the complex fixed points. The complex fixed points of \eqref{eq:mmob} are $\pm i c$, with 
$$g'(ic)=\left({c-ib\over c+ib}\right)^2,\ g'(-ic)=\left({c+ib\over c-ib}\right)^2.$$
Let $c+ib=ae^{i\theta}.$ Then  $g'(ic)=e^{-4i\theta}, g'(-ic)=e^{4i\theta},$ from which follows  the statement about  the angle of rotation. 
\end{proof}

\subsection{A  light source  outside an ellipse} 
Let us place a light  source $O$ outside an ellipse $C$. For each line through $O$ intersecting the interior of $C$ we consider the two billiard trajectories in the interior of $C$, whose initial rays are aligned with the line. One then finds analogues of the 2 conjectures and 2 theorems of this article, with ``4" replaced by ``2'' throughout: the $n$-th caustic by reflection  of these rays is tangent to $C$ at the contacts points with $C$ of the two tangents to $C$ through $O$, and has 2 cusps, located  on the hyperbola confocal with $C$ and passing through  $O$. 
See Figure \ref{fig:ext}.

\begin{figure}[ht]
\centering
\def\svgwidth{1\textwidth}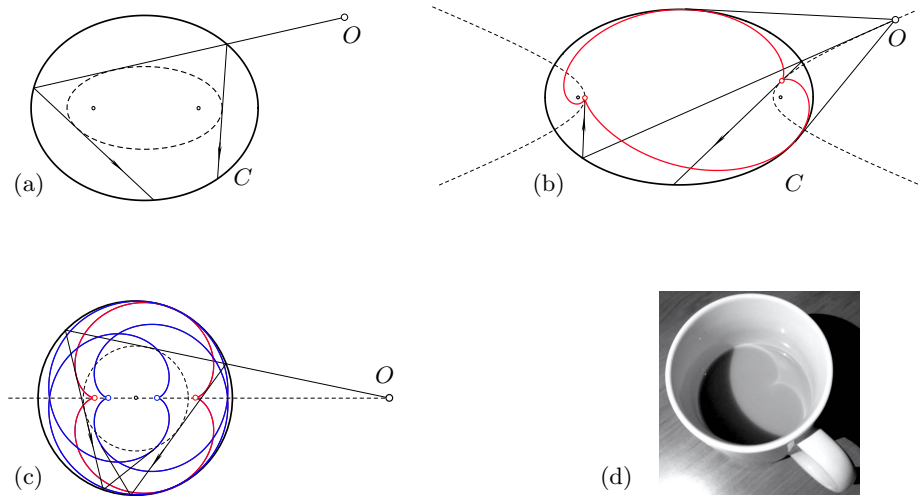
\caption{Caustics by reflection  from an external light source $O$: (a) each line through $O$, incident to the interior of $C$, produces 2 billiards trajectories, tangent to the same conic confocal to $C$; (b) the 1st  caustic by reflection off an ellipse, showing 2 cusps, lying on the confocal hyperbola through $O$. (c) The 1st two caustics by reflection off a circle, from an exterior light source $O$, showing two cusps for each caustic,  lying on the line  through $O$ and the center of the cicrcle. (d) A coffee cup ``half-caustic",  showing a single cusp. }
\label{fig:ext}
\end{figure}

\subsection{The complexity of the caustics by reflection} 
Figure \ref{complex} illustrates the observation that the complexity of the $n$-th caustic by reflection in an ellipse increases with $n$. There are many ways to measure ``complexity"; for example, one may consider the number of times that the caustic goes to infinity (these points correspond to the vertical tangents of the curve $T^n(\g)\subset \cL$).  It would be interesting to make conjectures in this direction.

\begin{figure}[ht]
\centering
\includegraphics[width=.3\textwidth]{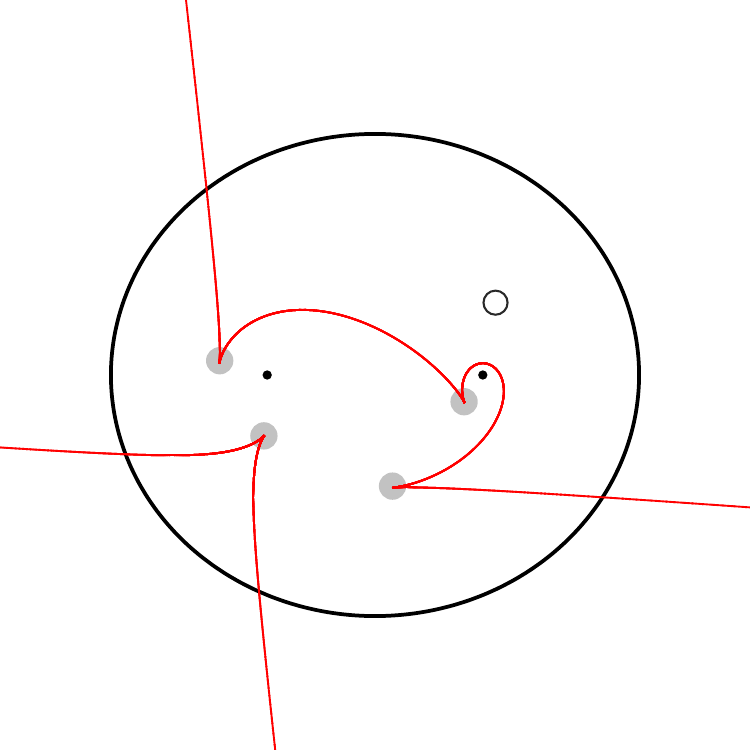}\ \ 
\includegraphics[width=.3\textwidth]{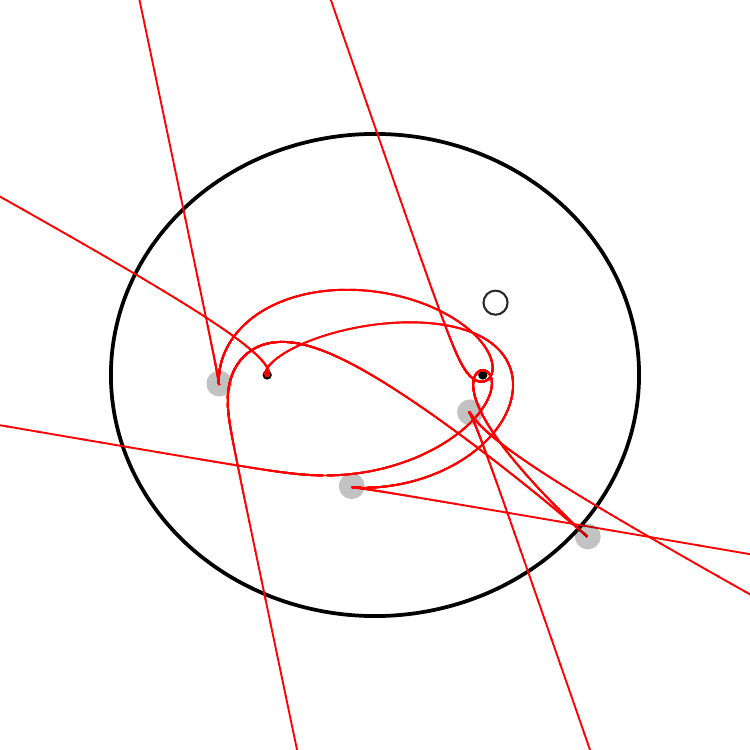}\ \ 
\includegraphics[width=.3\textwidth]{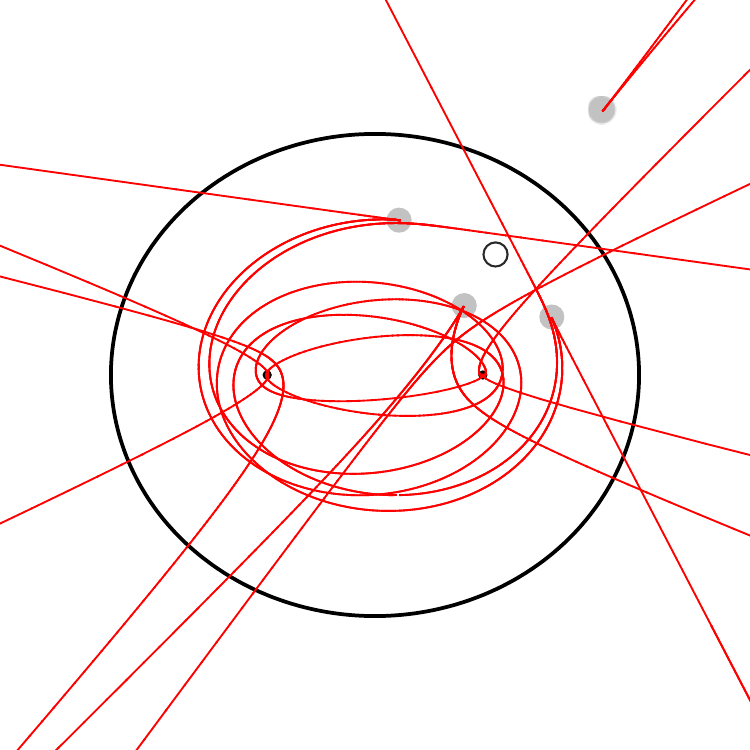}
\caption{The 2nd, 5th, and 8th caustics by reflection in an ellipse. The cusps are marked by gray circles.}
\label{complex}
\end{figure}

\subsection {Pseudo-integrable billiards} One may consider billiard tables bounded by arcs of confocal conics; such billiards were introduced in \cite{DR}.  Since confocal conics intersect at right angles, these billiard tables have angles that are multiples of $\pi/2$.  Figure \ref{pp} shows caustics by reflection in a table bounded by two confocal parabolas. Although four cusps still lie on the confocal parabolas that pass through the source of light, there are additional cusps, and their number increases with $n$.

\begin{figure}[ht]
\centering
\includegraphics[width=.3\textwidth]{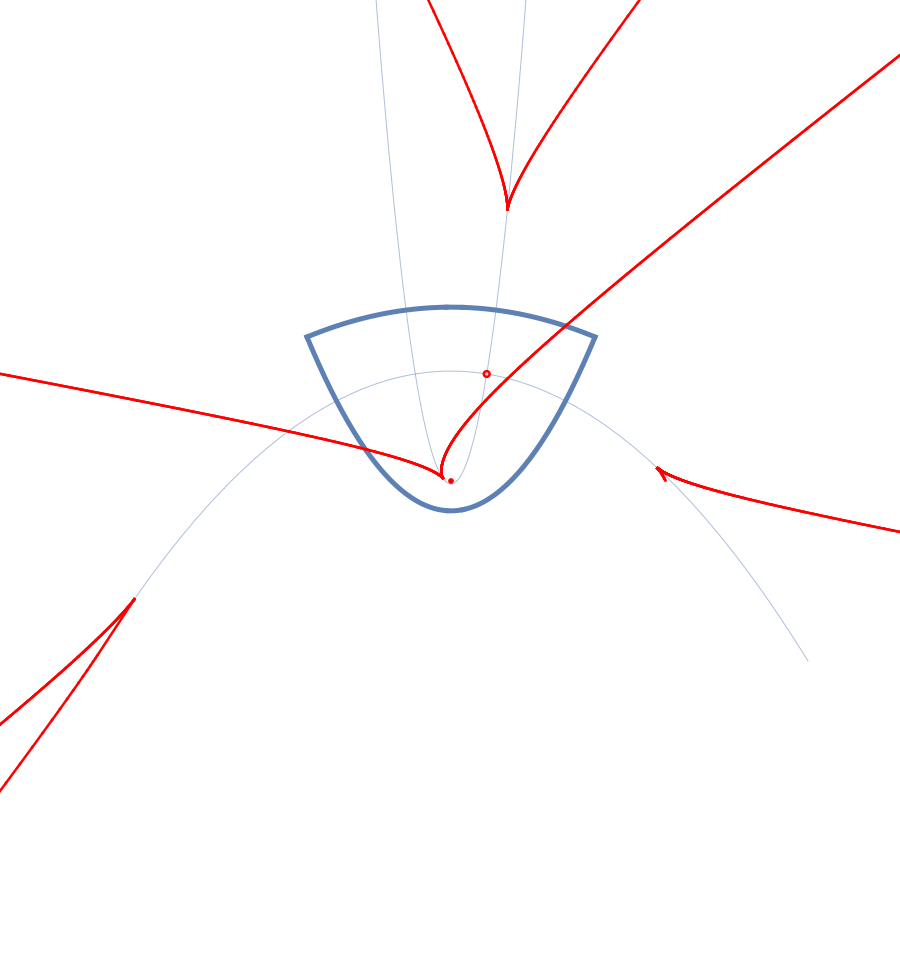}\ \ 
\includegraphics[width=.3\textwidth]{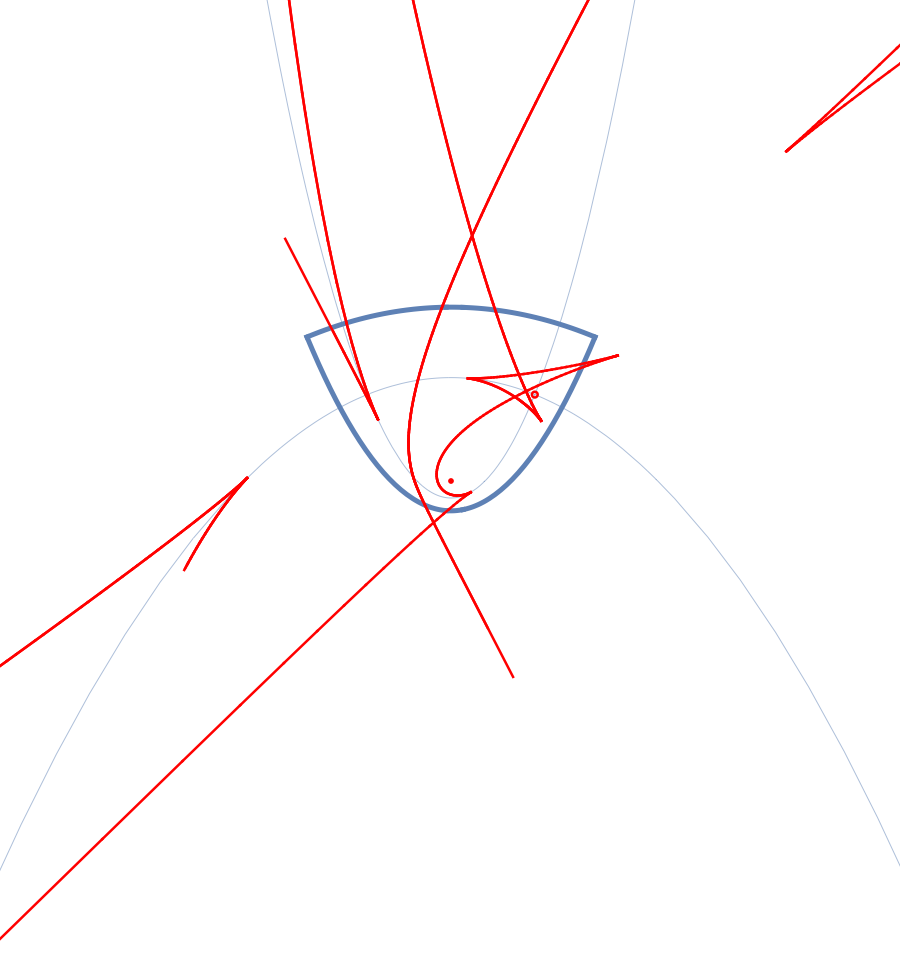}\ \ 
\includegraphics[width=.3\textwidth]{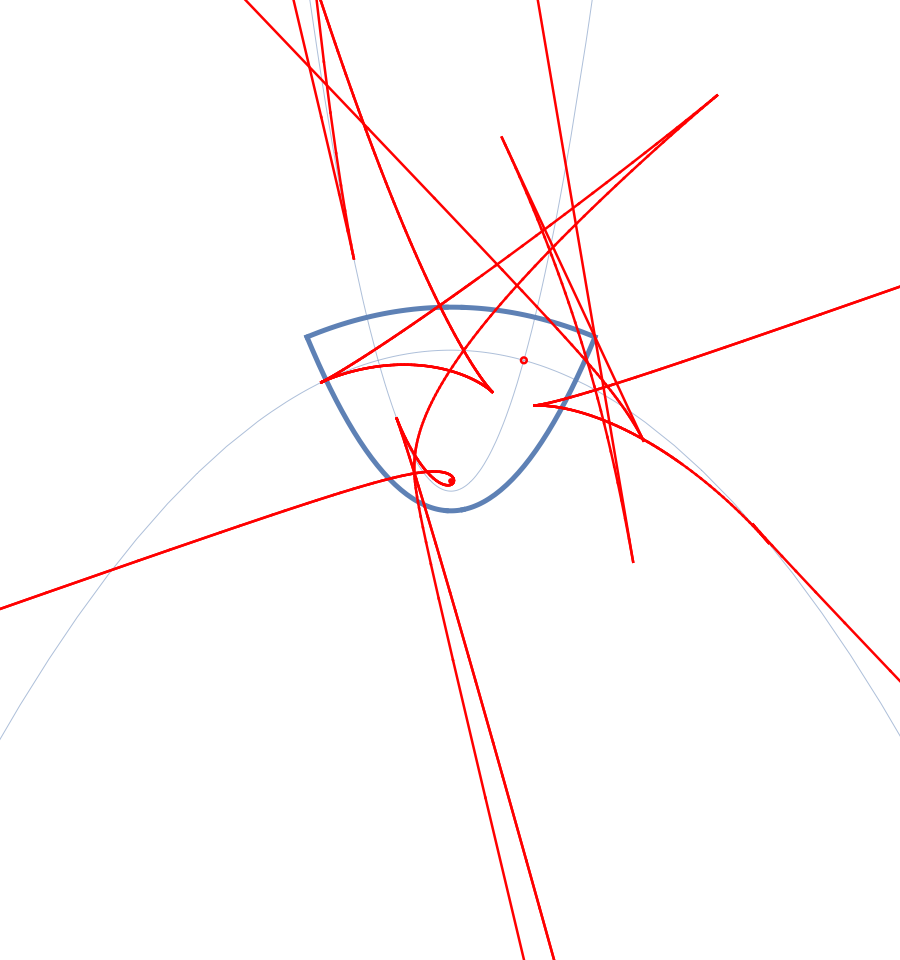}
\caption{The first three caustics by reflection in a table bounded by two confocal parabolas. }
\label{pp}
\end{figure}

\subsection{Caustics by refraction} 
One could extend the experimental study and make conjectures about caustics by refraction in ellipses. Cayley considered the first such caustic in the case of a circle in \cite{Ca}. See Figure \ref{fig:ref}, taken from p. 286 of Cayley's text.

\begin{figure}[ht]
\centering
\includegraphics[width=.9\textwidth]{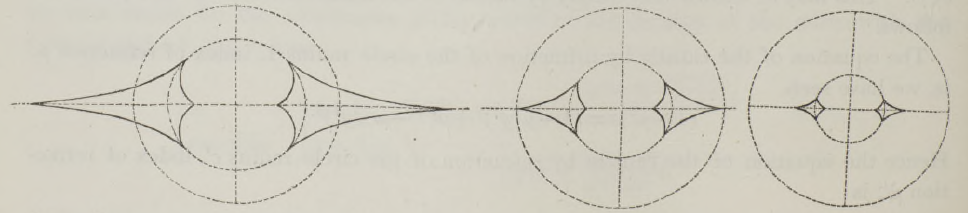}
\caption{Caustics by refraction of a parallel beam in a circle. The figures show 3 distinct values of the index  of refraction $\mu$ (from left to right): $1<\mu<2$, $\mu=2$, $\mu>2.$ The 4 cusps occur on the concentric circle of radius $1/\mu.$ }
\label{fig:ref}
\end{figure}

\end{document}